\documentclass[10pt,a4paper]{article}
\usepackage[a4paper]{geometry}
\usepackage{amssymb,latexsym,amsmath,amsfonts,amsthm,currvita}
\usepackage{graphicx,comment}
\usepackage{epsfig}
\usepackage{tikz}
\usepackage{subcaption}
\usepackage{comment}
\usepackage[T1]{fontenc}
\usepgflibrary{decorations.markings}
\usetikzlibrary{decorations.markings}
\usepackage{cancel}

\newcommand{\No}{\mathbb{N}_{\mathrm{odd}}}

\renewcommand{\l}{\left}
\renewcommand{\r}{\right}

\newcommand{\caliP}{\mathcal{P}}

\newcommand{\cj}{\overline}

\newcommand{\N}{\mathbb{N}}

\newcommand{\R}{\mathbb{R}}

\newtheorem{theorem}{Theorem}[section]
\newtheorem{lemma}[theorem]{Lemma}
\newtheorem{proposition}[theorem]{Proposition}

\newtheorem{corollary}[theorem]{Corollary}

\theoremstyle{definition}
\newtheorem{definition}[theorem]{Definition}

\theoremstyle{remark}

\newtheorem{remark}[theorem]{Remark}

\numberwithin{equation}{section}

\title{Higher-order asymptotics for the energy of greedy sequences on the unit circle}

\author{Abey L\'{o}pez-Garc\'{i}a \qquad Erwin Mi\~{n}a-D\'{i}az}

\date{\today}

\begin{document}

\maketitle

\begin{abstract}
For the Riesz and logarithmic energies, we consider a greedy sequence $(a_n)_{n=0}^\infty$ of points on the unit circle $S^1$ constructed in such a way that for every integer $N\geq 2$, the energy of the configuration $(a_0,\ldots,a_{N-2},x)$ attains its optimal value (say $E_N$) at $x=a_{N-1}$.  We derive an asymptotic expansion for $E_N$ in terms of certain bounded, oscillatory sequences $H_{N}$, $K_{N}$, and $R_{N}$ with a doubling periodicity property. In particular, we recover the results of \cite{LopMc1,LopWag} showing that after a proper translation and scaling of $E_N$, one is left with a sequence $T_N$ that is bounded and divergent. We show that the limit points of the sequence $T_N$ fill a closed interval. This follows from our asymptotic formulae and an analogous density result for the limit points of the sequences $H_{N}$, $K_{N}$, and $R_{N}$. We also give a new, simpler proof of density results obtained in \cite{LopMin} for the optimal values of the potential generated by a greedy sequence.
	
	\smallskip
	
	
	\smallskip
	
	
	\end{abstract}
\tableofcontents
\section{Introduction}

The Riesz $s$-kernel in the complex plane corresponding to a parameter $s\in\mathbb{R}$ is the function 
\[
k_{s}(z,w):=\begin{cases}
|z-w|^{-s}, & s\neq 0,\\
-\log |z-w|, & s=0.
\end{cases}
  \]
Given a configuration $\omega=(z_{1},\ldots,z_{N})$ of $N\geq 2$ distinct points on the unit circle $S^1=\{z\in\mathbb{C}:|z|=1\}$, the Riesz $s$-energy of $\omega$ is defined as
\[
E_{s}(\omega):=\sum_{1\leq i\neq j\leq N}k_{s}(z_{i},z_{j})=2\sum_{1\leq i<j\leq N}k_{s}(z_{i},z_{j}),
\]
and the $s$-potential generated by the $N\geq 1$ points $z_{1},\ldots,z_{N}$ is the function
\[
z\mapsto \sum_{\ell=1}^{N}k_{s}(z,z_{\ell}),\qquad z\in\mathbb{C}.
\]
In particular,
\[
E_{0}(\omega)=2\sum_{1\leq i<j\leq N}\log\frac{1}{|z_{i}-z_{j}|}.
\]

In the case $s\geq 0$, we say that a sequence $(a_{n})_{n=0}^{\infty}\subset S^{1}$ on the unit circle is a greedy $s$-energy sequence if it is constructed inductively as follows. The initial point $a_{0}$ is selected arbitrarily on $S^{1}$, and if $a_{0},\ldots, a_{N-1}$ have been selected, then the point $a_{N}\in S^{1}$ is chosen so that 
\begin{equation*}
	E_s((a_0,\ldots,a_{N-1},a_{N}))=\min_{z\in S^{1}}E_s((a_0,\ldots,a_{N-1},z))\qquad\mbox{for all}\,\,\,N\geq 1.
\end{equation*}

The potential generated by the first $N$ points of a greedy $s$-energy sequence $(a_{n})_{n=0}^{\infty}$ is the function 
\[
U_{N,s}(z):=\sum_{\ell=0}^{N-1}k_{s}(z,a_{\ell}),\qquad  N\geq 1.
\] 
From the identity 
\[
E_s(a_0,\ldots,a_{N-1},z)=2U_{N,s}(z)+E_s(a_0,\ldots,a_{N-1}),
\]
we see that the  potential   $U_{N,s}(z)$, as a function of  $z\in S^1$, is also minimized at $z=a_N$:
\begin{align}\label{defgreedyseq} 
	U_{N,s}(a_{N}) & =\min_{z\in S^{1}}U_{N,s}(z),\qquad s\geq 0.
\end{align} 
Since the function $U_{N,s}(z)$ being minimized is lower semicontinuous, the minimum is certainly attained. In the case $s<0$, the construction of a greedy $s$-energy sequence is identical but we replace the minimum in \eqref{defgreedyseq} by a maximum at each step. Note that the choice of a point $a_{n}$ in the sequence is in general not unique. In the logarithmic case $s=0$, the sequences obtained are commonly known as Leja sequences.

Due to the symmetry of the unit circle, greedy $s$-energy sequences on $S^{1}$ associated with a parameter $s>-2$ have identical geometric structure, independent of $s$. This follows from results in \cite{BC,CalviVan,LopMc1,LopWag}. It was first shown in \cite{BC} that the geometric structure of the first $N$ points of a Leja sequence on $S^{1}$ can be described in terms of the binary decomposition of $N$. This link with the binary representation of integers was the key to obtain a number of asymptotic results in \cite{LopMc1,LopWag} for the Riesz $s$-energy of sections of a greedy sequence, which were established in the range $s>-2$. The same link was exploited in \cite{LopMc2,LopMin} to obtain asymptotics for the extremal values $U_{N,s}(a_{N}) $. In the range $s<-2$, greedy sequences on $S^{1}$ are concentrated on two diametrically opposed points, and in the case $s=-2$, greedy sequences $(a_{n})_{n=0}^{\infty}$ are exactly those satisfying the condition $a_{2k+1}=-a_{2k}$ for each $k\geq 0$, see \cite{LopMc1}. So it follows that the main interest for the asymptotic analysis lies in the case $s>-2$.

If $(a_{n})_{n=0}^{\infty}$ is a greedy sequence on the unit circle associated with a parameter $s>-2$, we write
\[
\alpha_{N,s}:=(a_{0},\ldots,a_{N-1}),\qquad N\geq 1.
\]
 It follows from the main results obtained in \cite{LopMc1,LopWag} that after properly scaling and translating $E_{s}(\alpha_{N,s})$, one obtains a new sequence $(T_{N,s})_{N=2}^\infty$ that is \emph{bounded and divergent}, namely, the sequence
\begin{equation}\label{defTNs}
T_{N,s}:=\begin{cases}
E_{s}(\alpha_{N,s})-N^{2} I_{s}(\sigma), & -2<s<-1,\\[0.4em]
\displaystyle\frac{E_{-1}(\alpha_{N,-1})-N^{2}I_{-1}(\sigma)}{\log N}, & s=-1,\\[0.8em]
\displaystyle\frac{E_{0}(\alpha_{N,0})+N\log N}{N}, & s=0,\\[0.8em]
\displaystyle\frac{E_{s}(\alpha_{N,s})- N^{2}I_{s}(\sigma)}{N^{1+s}}, & -1<s<1,\,\,s\neq 0,\\[0.8em]
\displaystyle\frac{E_{1}(\alpha_{N,1})-\pi^{-1} N^{2} \log N}{N^{2}}, & s=1,\\[0.8em]
\displaystyle\frac{E_{s}(\alpha_{N,s})}{N^{1+s}}, & s>1,
\end{cases}
\end{equation}
where  (cf. \cite[Cor. A.11.4]{BorHarSaff})
\begin{align}\label{energyeqmeas}
I_{s}(\sigma):=\iint_{S^{1}\times S^{1}}|x-y|^{-s}\,d\sigma(x)\,d\sigma(y)=\frac{2^{-s}}{\sqrt{\pi}}\frac{\Gamma\left(\frac{1-s}{2}\right)}{\Gamma(1-\frac{s}{2})}
\end{align}
is the $s$-energy of the arclength measure $\sigma$ on $S^{1}$, normalized so that $\sigma(S^{1})=1$. The goal of the present paper is to gain a deeper understanding of the  asymptotic behavior of $T_{N,s}$ as $N\to\infty$. This behavior is primarily governed by certain arithmetic functions that we introduce next. 

Following \cite{LopWag}, for a vector $\vec{\theta}=(\theta_{1},\ldots,\theta_{p})$ with $p\in\mathbb{N}$ positive components $\theta_{k}>0$, and for any $s\in\mathbb{R}$, we define
\begin{equation}\label{def:H}
H(\vec{\theta};s):=\sum_{k=1}^{p}\theta_{k}^{s+1}+2(2^{s}-1)\sum_{k=1}^{p-1}\theta_{k}^{s}\sum_{j=k+1}^{p}\theta_{j},
\end{equation}
\begin{equation}\label{def:K}
K(\vec{\theta}):=2\log 2+\sum_{k=1}^{p}\theta_{k}^{2}\log(\theta_{k}/4)+2\sum_{k=1}^{p-1}\left(\sum_{j=k+1}^{p}\theta_{j}\right)\theta_{k}\log\theta_{k}.
\end{equation}
In the above formulas, the length $p$ of the vector $\vec{\theta}$ is arbitrary. We also introduce
\begin{equation}\label{def:Phi}
	R(\vec{\theta}):=-(2\log 2)\sum_{k=1}^{p}(k-1)\theta_{k}-\sum_{k=1}^{p}\theta_k\log\theta_{k}.
\end{equation}
As we will show, the functions defined in \eqref{def:H}--\eqref{def:Phi} play the leading role in the asymptotics of $T_{N,s}$.

For an integer $N\in\mathbb{N}$ with binary decomposition 
\begin{equation*}
	N=2^{n_{1}}+2^{n_{2}}+\cdots+2^{n_{p}},\qquad n_{1}>n_{2}>\cdots>n_{p}\geq 0,
\end{equation*}
we write
\begin{equation*}
\eta(N) :=\left(\frac{2^{n_{1}}}{N},\frac{2^{n_{2}}}{N},\ldots,\frac{2^{n_{p}}}{N}\right).
\end{equation*}
Observe that this vector function has the doubling periodicity property
\begin{equation}\label{doubperiod1}
\begin{aligned}
\eta(N) & =\eta(2N).
\end{aligned}
\end{equation}
Note also that $H(\eta(N);0)=H(\eta(N);1)=1$ for all $N\in\mathbb{N}$.

In the following result, and throughout the paper, $\zeta(s)$ denotes the Riemann zeta function, and $\gamma=\lim_{N\rightarrow\infty}(\sum_{k=1}^{N}\frac{1}{k}-\log N)$ is the Euler-Mascheroni constant.

\begin{theorem}\label{theo:leadasympTNs}
The following asymptotic formulae hold true as $N\rightarrow\infty$. If $s>-1$ and $s\neq 0, 1$, then
	\begin{equation}\label{expTNs}
		T_{N,s}=\frac{2\zeta(s)}{(2\pi)^{s}}H(\eta(N);s)+\begin{cases}
			O(N^{-1-s}), & -1<s<1,\,\,s\neq 0,\\
			O(N^{1-s}), & 1<s<3,\,\,s\neq 2,\\
			O(N^{-2}), & s=2,\,\,\mbox{or}\,\,s>3,\\
			O(N^{-2}\log N), & s=3.
		\end{cases}
	\end{equation}
	If $s=-1$, then
	\begin{equation}\label{expTN-1}
		T_{N,-1}=-\frac{\pi}{3}\frac{H(\eta(N);-1)}{\log N}+O((\log N)^{-1}).
	\end{equation}
		If $s=1$, then
	\begin{equation}\label{expTN1}
		T_{N,1}=\frac{1}{\pi}(\gamma+\log(2/\pi)+K(\eta(N)))+O(N^{-2}).
	\end{equation}
	If $s=0$, for each integer $N\geq 2$ we have
	\begin{equation}\label{TN0Phi}
		T_{N,0}=R(\eta(N)).
	\end{equation}
\end{theorem}

By the periodicity property \eqref{doubperiod1},  Theorem \ref{theo:leadasympTNs} readily yields the following 
\begin{corollary}\label{cor1}
	For any $s\geq -1$, the following asymptotic doubling periodicity holds:
	\[
	\lim_{N\rightarrow\infty}(T_{2N,s}-T_{N,s})=0.
	\]
\end{corollary}

Theorem \ref{theo:leadasympTNs} is a consequence of the stronger  
 Theorems \ref{firstthm-expansion} and \ref{secondthm-expansion} of Section \ref{section:EnergyExpansion}, which further provide, for every $s\geq -1, \ s\not=0$, a finite asymptotic expansion of $E_{s}(\alpha_{N,s})$ in terms of the functions $H$ and $K$ of  \eqref{def:H}--\eqref{def:K}. The larger the value of $s$, the more asymptotic terms that the expansion is able to provide. 
 
 Because of the primary role played by the functions $H$, $K$, and $R$ in the asymptotic behavior of the $s$-energy of greedy sequences, we make in Sections \ref{limit-points-functions-H-K-R} through \ref{continuity-functions-H-K-R} an extensive study of the properties of these functions with the ultimate goal of proving the following result. 
 \begin{theorem}\label{theo:density-interval}
 For every $s>-1$, the set of limit points of the sequence $(H(\eta(N);s))_{N=1}^\infty$ is a closed interval. The same is true of the sequences $(K(\eta(N)))_{N=1}^\infty$ and $(R(\eta(N)))_{N=1}^\infty$. 
 Therefore, it follows from Theorem \ref{theo:leadasympTNs} that for every $s> -1$, the set of  limit points of the sequence $(T_{N,s})_{N=2}^\infty$ is a closed interval.
 \end{theorem}
 
 \begin{remark}
 	Regarding the case $s=-1$ not covered in Theorem \ref{theo:density-interval}, we do not know whether  the set of limit points of the sequence $(T_{N,-1})_{N=2}^{\infty}$ is an interval. However, we can indicate some estimates. From \eqref{expTN-1} and  \eqref{ineqHminusone} we obtain that
 	\[
 	\liminf_{N\rightarrow\infty} T_{N,-1}\geq -\frac{\pi}{3 \log 2}.
 	\]
 	This and formula (3.78) in \cite{LopMc1} yield  the bounds
 	\[
 	-\frac{\pi}{3 \log 2}\leq \liminf_{N\rightarrow\infty} T_{N,-1}\leq -\frac{\pi}{9 \log 2}.
 	\]
 	Formula (3.77) in \cite{LopMc1} indicates that $\limsup_{N\rightarrow\infty} T_{N,-1}=0$.
 \end{remark}

Let $\Delta_{s}$ denote the set of limit points of the sequence $(H(\eta(N);s))_{N=1}^{\infty}$ for $s>-1$, $s\neq 0, 1$. Let $\Delta_{0}$ and $\Delta_{1}$ denote the sets formed by the limit points of the sequences $(R(\eta(N)))_{N=1}^{\infty}$ and $(K(\eta(N)))_{N=1}^{\infty}$, respectively. We prove Theorem \ref{theo:density-interval} by showing that these sets $\Delta_s$ are the range of certain continuous functions defined on the interval $[1/2,1]$. Given $x\in [1/2,1]$, the number $1/x$ has at least one (and at most two, one finite and one infinite)  binary expansion of the form
\begin{align}\label{ecua7}
\frac{1}{x}=\frac{1}{2^{k_{1}}}+\frac{1}{2^{k_{2}}}+\cdots+\frac{1}{2^{k_{j}}}+\cdots,\qquad 0=k_1<k_2<\cdots,
\end{align}
which we use to define
\[
	\vec{\theta}(x)=\left(x,\frac{x}{2^{k_{2}}},\frac{x}{2^{k_{3}}},\ldots\right),\qquad x\in [1/2,1].
\] 
The functions $H(\vec{\theta};s)$, $K(\vec{\theta})$, and $R(\vec{\theta})$ of \eqref{def:H}--\eqref{def:Phi} have  natural extensions to a broader class of infinite vectors $\vec{\theta}$ that, in particular, contains the vectors  $\vec{\theta}(x)$, $x\in [1/2,1]$ (see Definition \ref{def:extensionofHKR}). With these extensions at hand, we define for all $s>-1$ and $x\in [1/2,1]$
\[
\mathcal{H}(x,s):= H(\vec{\theta}(x);s),\qquad \mathcal{K}(x):=K(\vec{\theta}(x)),\qquad \mathcal{R}(x):=R(\vec{\theta}(x)).
\]  

It is remarkable that these functions are well-defined (i.e. independent of the choice of the vector $\vec{\theta}(x)$ for each $x$)  and \emph{continuous} in the variable $x$. We shall demonstrate in Section \ref{continuity-functions-H-K-R} that 
\[
\Delta_0=\{\mathcal{R}(x):1/2\leq x\leq 1\},
\]
\[
\Delta_1=\{\mathcal{K}(x):1/2\leq x\leq 1\}=[0,\kappa],
\]
and 
\[
\Delta_s=\{\mathcal{H}(x,s):1/2\leq x\leq 1\}=\begin{cases}
	[d_s,1],&\quad 0<s<1,\\
	[1,d_s], &\quad -1<s<0\  or \  s>1,
	\end{cases}
\]
with $\kappa=\max_{1/2\leq x\leq 1}\mathcal{K}(x)>0$ and
\begin{align*}
d_s={} &\begin{cases}\min_{1/2\leq x\leq 1}\mathcal{H}(x,s)<1,& \quad 0<s<1,\\
\max_{1/2\leq x\leq 1}\mathcal{H}(x,s)>1, &\quad -1<s<0\ or\  s>1.
\end{cases}
\end{align*}
The interval $\Delta_0$ can be given explicitly. Indeed, in \cite[Thm. 1.2]{LopWag} it was proved that	\begin{equation}\label{boundsTN0}
	0\leq T_{N,0}<\log(4/3),\qquad \mbox{for all}\,\,N\geq 2,
\end{equation}
where the upper bound is best possible since 
\[
\limsup_{N\rightarrow\infty} T_{N,0}=\log(4/3).
\]
On the other hand, the logarithmic energy of $N$ equally spaced points on the unit circle equals $-N\log N$ and the points in the configuration $\alpha_{2^{k},0}$ are equally spaced, so that the lower bound in \eqref{boundsTN0} is attained for every $N=2^{k}$, $k\in\mathbb{N}$. Therefore, since $T_{N,0}=R(\eta(N))$ for all $N\geq 2$,  we deduce that 
\[
\Delta_{0}=[0,\log(4/3)].
\]
\begin{figure}
	\begin{center}	\includegraphics[scale=.5]{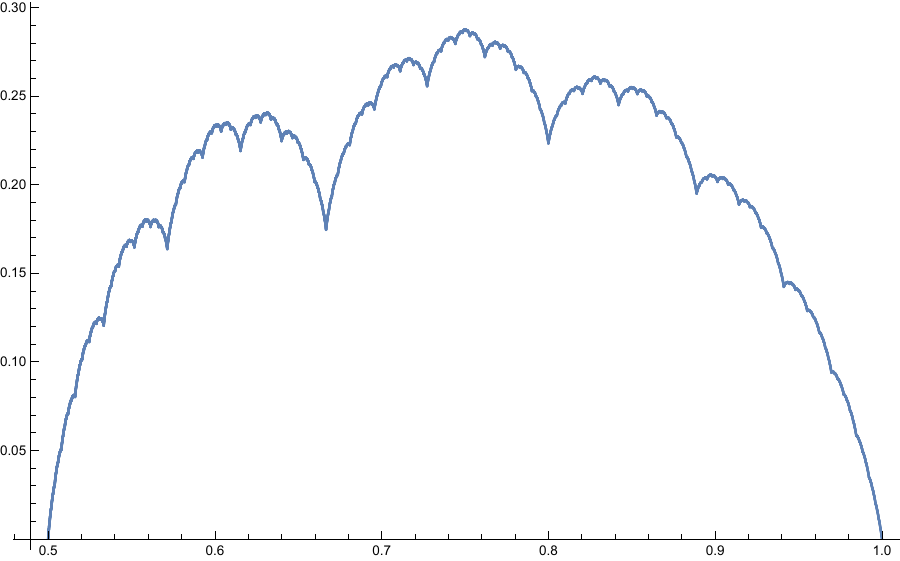}
	\end{center}
	\caption{Plot of $\mathcal{R}(x)$ for $x\in \mathcal{P}_{16}$. } 
\end{figure}
\begin{figure}
	\begin{center}	\includegraphics[scale=.5]{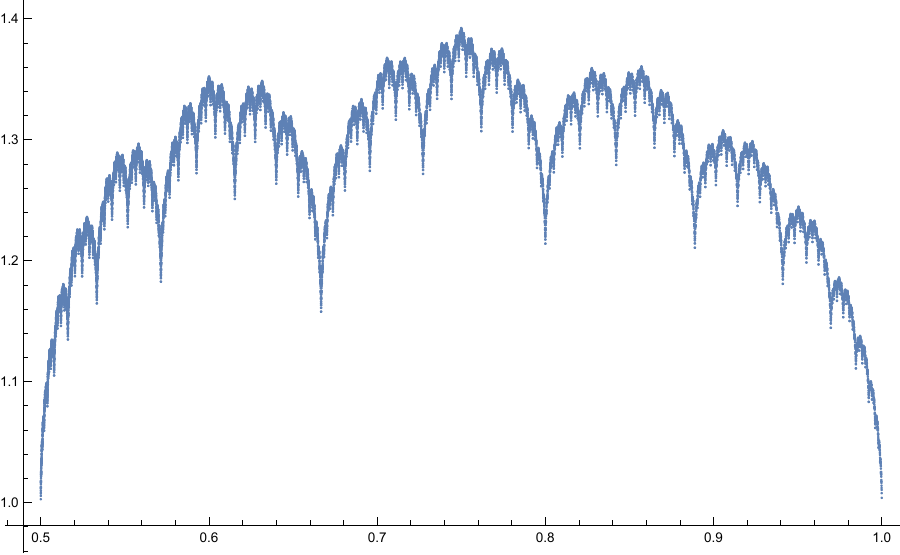}
	\end{center}
	\caption{Plot of $\mathcal{H}(x,-1/2)$ for $x\in \mathcal{P}_{16}$. This yields an approximate value of $d_{-1/2}\approx d_{-1/2,16}=1.3924456\ldots$.} 
\end{figure}

\begin{figure}
	\begin{center}	\includegraphics[scale=.5]{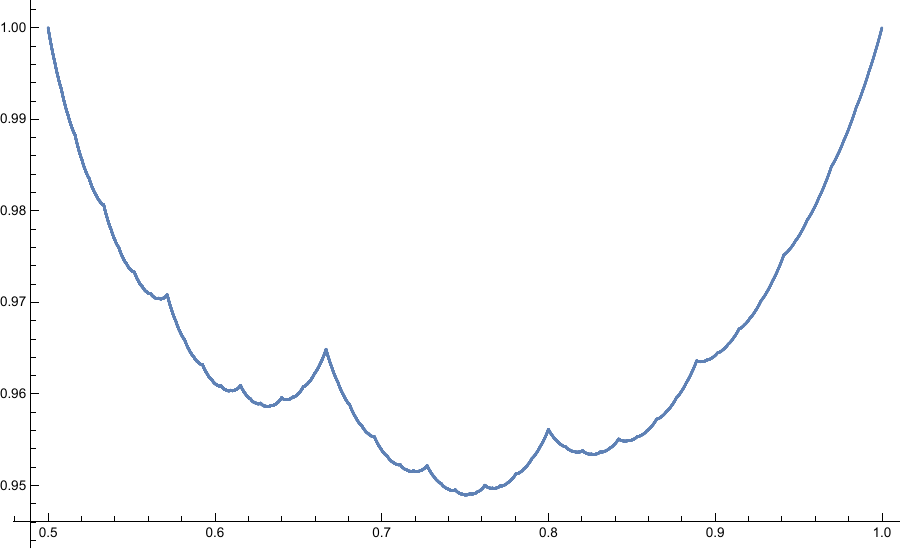}
	\end{center}
	\caption{Plot of $\mathcal{H}(x,1/3)$ for $x\in\mathcal{P}_{16}$. This yields an approximate value of $d_{1/3}\approx d_{1/3,16}=0.9489466\ldots$.} 
\end{figure}

\begin{figure}
	\begin{center}	\includegraphics[scale=.5]{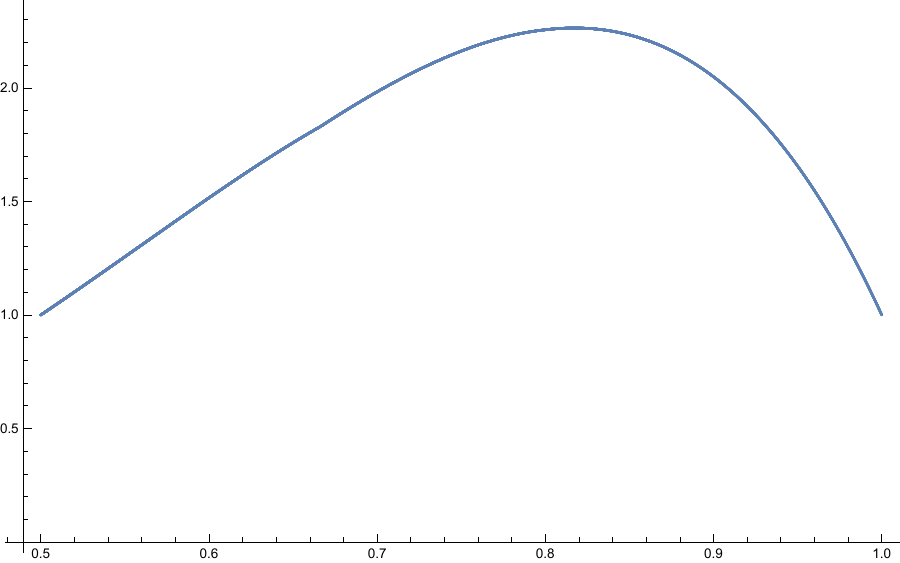}
	\end{center}
	\caption{Plot of $\mathcal{H}(x,7/2)$ for $x\in \mathcal{P}_{16}$. This yields an approximate value of $d_{7/2}\approx d_{7/2,16}=2.2640277\ldots$.} 
\end{figure}

\begin{figure}
	\begin{center}	\includegraphics[scale=.5]{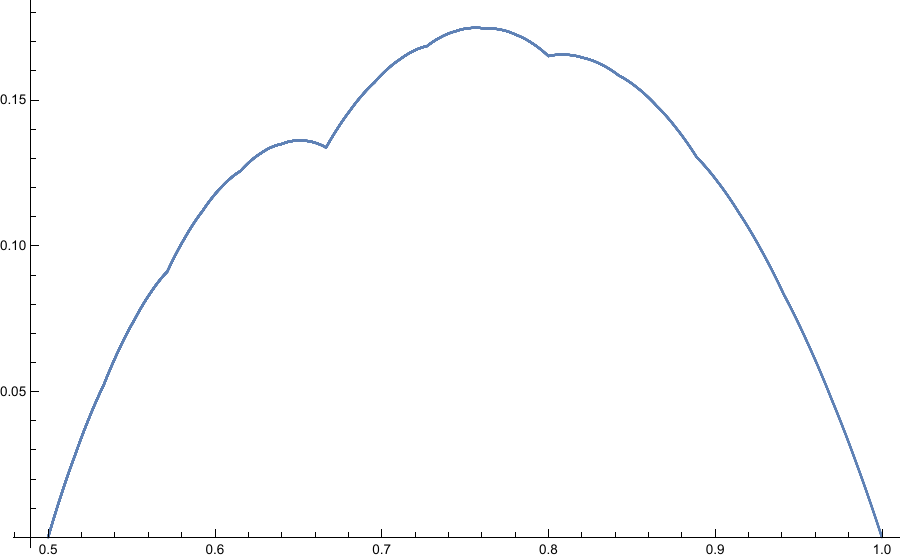}
	\end{center}
	\caption{Plot of $\mathcal{K}(x)$ for $x\in\mathcal{P}_{16}$. This yields an approximate value of $\kappa\approx0.1747397$.} 
\end{figure}
The exact values of the constants $\kappa$  and $d_{s}$ are unknown to us, but they can  be computed approximately as follows. For every integer $M\geq 1$, let $\mathcal{P}_M$ be the collection of points
\begin{align}\label{def:xMn}
x_{M,n}:=\frac{2^M}{2^M+2n+1},\qquad  n=0, 1,\ldots,2^{M-1}-1.
\end{align}
Taken collectively, the sets $\mathcal{P}_M$ contain  precisely those points $x\in(1/2,1)$ such that $1/x$ has a binary representation of the form \eqref{ecua7} with finitely many terms. Let us define 
\[
d_{s,M}:=\begin{cases}\min_{x\in \mathcal{P}_M}\mathcal{H}(x,s), & \ 0<s<1,\\
	\max_{x\in \mathcal{P}_M}\mathcal{H}(x,s), & \ -1<s<0\ or \ s\geq 1,
		\end{cases}
\] 
and 
\[
\kappa_{M}:=\max_{x\in\mathcal{P}_{M}}\mathcal{K}(x).
\]
Because the distance between two consecutive points of $\mathcal{P}_M\cup \{1/2,1\}$ is less than $2^{-M+1}$, it follows that 
\[
\lim_{M\to\infty}d_{s,M}=d_s,\qquad \lim_{M\rightarrow\infty} \kappa_{M}=\kappa.
\]
Estimates on the speed of convergence in the first limit are provided by our next result.
\begin{theorem}\label{approxthm}For all $M\geq 1$, we have
\begin{align}\label{ecua12}
	0\leq d_{s,M}-	d_s\leq \frac{2^s}{2^{M-1}},\qquad 0<s<1,
\end{align}	
\begin{align}\label{ecua13}
	0\leq d_s-d_{s,M}\leq \frac{2^s}{2^{M-1}},\qquad s>1,
\end{align}	
\begin{align}\label{ecua14}
	0\leq d_s-d_{s,M}\leq \frac{2^{(1-M)(s+1)}}{(\log 2)(2^{s+1}-1)},\qquad -1<s<0.
\end{align}	
\end{theorem}

In the last section of the paper we briefly revisit some aspects of the asymptotic  behavior as $N\to\infty$ of the extremal value $U_{N,s}(a_N)$. This behavior has  been previously  investigated in \cite{LopMc1,LopMc2,LopMin,LopSaff}. In particular, it was shown in \cite{LopMc1,LopMc2} that for every $s>-2$, the sequence $(F_{N,s})_{N=1}^\infty$  defined as 
\begin{align}\label{defF_N,s}\begin{split}
		F_{N,s}:=\begin{cases}
		 U_{N,s}(a_{N})-N I_{s}(\sigma), & -2<s<0,\\[0.3em]
			\displaystyle -\frac{U_{N,0}(a_{N})}{\log(N+1)}, & s=0,\\[0.9em]
			\displaystyle\frac{U_{N,s}(a_N)-NI_s(\sigma)}{N^s},& 0<s<1,\\[0.5em]
			\displaystyle\frac{U_{N,1}(a_{N})-\pi^{-1} N\log N}{N}, & s=1,\\[0.5em]
			\displaystyle\frac{U_{N,s}(a_N)}{N^s}, & s>1
		\end{cases}
	\end{split}
\end{align}
is bounded and divergent. More recently in \cite[Theorem 1.2]{LopMin}, it was shown that if  $s>0$ and $s\not=1$, then (as $N\to\infty$)
\begin{align}\label{thirdcase}
	\begin{split}
		F_{N,s}	={} & (2^{s}-1)\frac{2\zeta(s)}{(2\pi)^s}G(\eta(N);s)+\begin{cases}
			O(N^{-s}), & 0<s<1,\\
			O(N^{1-s}),& 1<s<3,\ s\not=2,\\
			O(N^{-2}),	& s=2, \ \mathrm{or}\   s>3,\\
			O(N^{-2}\log N), & s=3,
		\end{cases}
	\end{split}
	\end{align}
while if $s=1$, then 
\begin{align}\label{secondcase}
	\begin{split}
		F_{N,1}= {} &
		\frac{1}{\pi}(\gamma+\log(8/\pi) +\Lambda(\eta(N)))+O(N^{-1}), 	
	\end{split}
\end{align} 
where  
\begin{align}\label{def:G-Lambda-functions}
	G(\vec{\theta};s):={} \sum_{k=1}^{p}\theta_{k}^{s}\quad \mathrm{and}\quad 	\Lambda(\vec{\theta}):={} \sum_{k=1}^{p}\theta_{k}\log\theta_{k}
\end{align}
are functions defined for every vector $\vec{\theta}=(\theta_{1},\ldots,\theta_{p})$ of   positive components. This result, which can be viewed as a counterpart to Theorem \ref{theo:leadasympTNs}, is obtained as a byproduct of a finite asymptotic expansion for $U_{N,s}(a_N)$  given in  Theorems 1.4 and 1.5 of \cite{LopMin},  which in turn can be viewed as the counterparts of  Theorems \ref{firstthm-expansion} and \ref{secondthm-expansion} of Section \ref{section:EnergyExpansion}. Despite these similarities, the study of the  sequence $(T_{N,s})$ is visibly more difficult than that of the sequence $(F_{N,s})$ because  the expressions that define the functions  $H$, $K$, and $R$ are substantially more complex than those that define the functions $G$ and $\Lambda$.

It was shown in \cite{LopMin} that for any $s\geq 0$, the limit points of the sequence $(F_{N,s})_{N=1}^\infty$ fill a closed  interval. For the case $s>0$, we will provide in Section \ref{sect:density-optimal-potential} a new, more transparent proof of this result based on the same ideas we have developed for the proof of Theorem \ref{theo:density-interval}. In that same section we show that in the range $-2<s<0$, the sequence $(F_{N,s})_{N=1}^{\infty}$ is Ces\`{a}ro summable.

\section{Expansion of the energy}\label{section:EnergyExpansion}
Of special importance for our asymptotic analysis is the energy
\begin{equation}\label{def:LsN}
	\mathcal{L}_{s}(N):=E_{s}((z_{1,N},\ldots,z_{N,N}))
\end{equation}
of the configuration formed by the $N$th roots of unity
\[
z_{k,N}:=\exp(2\pi i(k-1)/N),\qquad 1\leq k\leq N.
\]
It is easy to see that
\[
\mathcal{L}_{s}(N)=\begin{cases}
	2^{-s} N\sum_{k=1}^{N-1}\left(\sin\frac{\pi k}{N}\right)^{-s} & s\neq 0,\\
	-N\log N & s=0.
\end{cases}
\]
In the case $N=1$, we understand $\mathcal{L}_{s}(1)=0$.

An asymptotic expansion for  $\mathcal{L}_{s}(N)$ as $N\to\infty$ was derived in  \cite{BHS}. This expansion   involves the Riemann zeta function $\zeta(s)$, the quantity 
\[
v(s):=\frac{2^{-s}\,\Gamma(\frac{1-s}{2})}{\sqrt{\pi}\,\Gamma(1-\frac{s}{2})}, \qquad s\in \mathbb{C}\setminus \{1,3,5,\ldots\},
\]
as well as the coefficients $\beta_n(s)$ of the Maclaurin expansion of the function $\mathrm{sinc}^{-s}z$:
\[
\mathrm{sinc}^{-s}\,z:=\left(\frac{\sin(\pi z)}{\pi z}\right)^{-s}=\sum_{n=0}^{\infty}\beta_{n}(s) z^{2n},\qquad |z|<1.
\]
Note that  $\beta_{0}(s)=1$, and that in view of \eqref{energyeqmeas},  $v(s)=I_{s}(\sigma)$ for $-2<s<1$, $s\not=0$. This function $v(s)$ is analytic in the indicated domain, and vanishes at every positive and even integer. 

In this section we use the following notation. $\mathbb{N}_{\mathrm{odd}}$ denotes the set of all odd positive integers, $\lfloor x\rfloor$ denotes the integer part of $x$, and $(x)_{n}$ denotes the Pochhammer symbol. 

It was proved in \cite[Thm. 1.1]{BHS} that for $s\geq -1$ such that $s\neq 0$, $s\notin\mathbb{N}_{\mathrm{odd}}$, and for any integer $J\geq 0$,
\begin{equation}\label{expLsN}
	\mathcal{L}_{s}(N)=v(s) N^{2}+\frac{2}{(2\pi)^{s}}\sum_{j=0}^{J}\beta_{j}(s)\,\zeta(s-2j)\,N^{s-2j+1}+O(N^{s-2J-1}),\qquad N\rightarrow\infty,
\end{equation} 
while if $s$ is a positive even integer, then $v(s)=0$ and \eqref{expLsN} holds in the exact form
\begin{equation}\label{exactformLsN}
	\mathcal{L}_{s}(N)=\frac{2}{(2\pi)^{s}}\sum_{j=0}^{s/2}\beta_{j}(s)\,\zeta(s-2j)\,N^{s-2j+1}.
\end{equation}

We shall use \eqref{expLsN} and \eqref{exactformLsN} to obtain an asymptotic expansion for the energy 	$E_{s}(\alpha_{N,s})$ of the first $N$ points $\alpha_{N,s}=(a_0,\ldots,a_{N-1})$ of a greedy $s$-energy sequence. Suppose a positive integer $N\in\mathbb{N}$ has the binary decomposition
\begin{equation}\label{bindecomp}
	N=2^{n_{1}}+2^{n_{2}}+\cdots+2^{n_{p}},\qquad n_{1}>n_{2}>\cdots>n_{p}\geq 0,
\end{equation}
in which case we write 
\begin{equation}\label{def:arithfunctau}
	\begin{aligned}
		\tau_{b}(N) & :=p
	\end{aligned}
\end{equation}
to denote the number of $1$'s in the binary representation of $N$. Note that this arithmetic function satisfies the doubling periodicity property
\begin{equation}\label{doubperiod}
	\begin{aligned}
		\tau_{b}(N) & =\tau_{b}(2N).
	\end{aligned}
\end{equation}
In \cite{LopMc1,LopWag} it was shown that the energy $	E_{s}(\alpha_{N,s})$ can be expressed in terms of the binary decomposition \eqref{bindecomp} of $N$ by the formula 
\begin{equation}\label{energybin}
	E_{s}(\alpha_{N,s})=\sum_{k=1}^{p-1}\left(\sum_{t=k+1}^{p} 2^{n_{t}-n_{k}}\right)\mathcal{L}_{s}(2^{n_{k}+1})+\sum_{k=1}^{p}\left(1-\sum_{t=k+1}^{p} 2^{n_{t}-n_{k}+1}\right)\mathcal{L}_{s}(2^{n_{k}}).
\end{equation}
In order to extricate the asymptotic behavior of $	E_{s}(\alpha_{N,s})$ from this representation, we first need to establish a couple of propositions. Recall that we used  \eqref{bindecomp} to define  \begin{equation}\label{def:arithfunc}
	\begin{aligned}
		\eta(N) & :=\left(\frac{2^{n_{1}}}{N},\frac{2^{n_{2}}}{N},\ldots,\frac{2^{n_{p}}}{N}\right).
	\end{aligned}
\end{equation}

\begin{proposition}\label{propboundedseq}
For each $s>-1$, the sequence $\{H(\eta(N);s)\}_{N\geq 1}$
is bounded. The sequences $\{K(\eta(N))\}_{N\geq 1}$ and $\{R(\eta(N))\}_{N\geq 1}$ are also bounded.
\end{proposition}
\begin{proof}
Let $N\in\mathbb{N}$ have the binary decomposition \eqref{bindecomp}, and let
\[
\theta_{k}=\frac{2^{n_{k}}}{N},\qquad 1\leq k\leq p.
\] 
Observe that $\sum_{k=1}^p\theta_{k}=1$.  
Suppose first that $s>0$. Since $2^{s}-1>0$, for each $1\leq k\leq p$ we obtain
\begin{equation}\label{estim1}
0< 2(2^{s}-1)\left(\sum_{j=k+1}^{p}\theta_{j}\right)+\theta_{k}< 2(2^{s}-1)+1=2^{s+1}-1.
\end{equation}
If $s\geq 1$, then obviously $\sum_{k=1}^{p}\theta_{k}^{s}\leq \sum_{k=1}^{p}\theta_{k}=1$. Hence, using \eqref{def:H}, we obtain
\begin{align*}
0< H(\eta(N);s)< 2^{s+1}-1,\qquad s\geq 1.
\end{align*}
Assume now that $0<s<1$. By Theorem 1.1 in \cite{LopMin} we have
\[
\sum_{k=1}^{p}\theta_{k}^{s}\leq \frac{1}{2^{s}-1},
\]
which together with \eqref{estim1} imply that 
\[
0<H(\eta(N);s)< \frac{2^{s+1}-1}{2^{s}-1}.
\]
Now suppose that $-1<s<0$. For each $1\leq k\leq p$ we have
\begin{equation}\label{ineqtheta}
\sum_{j=k+1}^{p}\theta_{j}=\frac{1}{N}\sum_{j=k+1}^{p}2^{n_{j}}\leq\frac{1}{N}\sum_{m=0}^{n_{k+1}}2^{m}=\frac{2^{n_{k+1}+1}-1}{N}\leq \frac{2^{n_{k}}-1}{N}< \theta_{k}.
\end{equation}
Since $-1<2(2^{s}-1)<0$, it follows that
\[
0<\theta_{k}+2(2^{s}-1)\sum_{j=k+1}^{p}\theta_{j}\leq \theta_{k}.
\]
Therefore
\[
0<H(\eta(N);s)=\sum_{k=1}^{p}\theta_{k}^{s}\left(2(2^{s}-1)\left(\sum_{j=k+1}^{p}\theta_{j}\right)+\theta_{k}\right)\leq \sum_{k=1}^{p}\theta_{k}^{s+1}.
\]
Also note that
\begin{equation}\label{boundforthetak}
\theta_{k}=\frac{2^{n_{k}}}{N}\leq \frac{2^{n_{k}}}{2^{n_{1}}}\leq 2^{-(k-1)}
\end{equation}
and so
\[
\sum_{k=1}^{p}\theta_{k}^{s+1}\leq \sum_{k=1}^{p}2^{-(k-1)(s+1)}< \sum_{m=0}^{\infty}2^{-(s+1)m}=\frac{2^{s+1}}{2^{s+1}-1}.
\]
We conclude that
\begin{align}\label{ecua6}
0<H(\eta(N);s)< \frac{2^{s+1}}{2^{s+1}-1},\qquad -1<s<0.
\end{align}
Observe that $H(\eta(N);0)=1$ for all $N\in\mathbb{N}$.

Since $0<\theta_{k}\leq 1$, we have $\log\theta_{k}\leq 0$, and by Theorem 3.1 in \cite{LopMin} we know that $\sum_{k=1}^{p}\theta_{k}\log\theta_{k}>-\log 4$, hence
\begin{equation}\label{knownbound}
\sum_{k=1}^{p}\theta_{k}|\log\theta_{k}|<\log 4.
\end{equation}
Therefore
\[
\left|\sum_{k=1}^{p}\theta_{k}^{2}\log(\theta_{k}/4)\right|\leq \sum_{k=1}^{p}\theta_{k}^{2}\,|\log\theta_{k}|+\log(4)\sum_{k=1}^{p}\theta_{k}^{2}\leq \sum_{k=1}^{p}\theta_{k}\,|\log \theta_{k}|+\log(4)\sum_{k=1}^{p}\theta_{k}\leq 2\log 4.
\]
Using \eqref{ineqtheta}, we also have
\[
\left|\sum_{k=1}^{p-1}\sum_{j=k+1}^{p}\theta_{j}\,\theta_{k}\log\theta_{k}\right|=\sum_{k=1}^{p-1}\sum_{j=k+1}^{p}\theta_{j}\,\theta_{k}|\log\theta_{k}|\leq \sum_{k=1}^{p-1}\theta_{k}^{2}\,|\log\theta_{k}|<\log 4.
\]
So it follows from \eqref{def:K} that $|K(\eta(N))|\leq 5 \log 4$.

By \eqref{boundforthetak}, we have
\[
\sum_{k=1}^{p}(k-1)\theta_{k}\leq \sum_{k=1}^{p}\frac{k-1}{2^{k-1}}<2,
\]
which together with \eqref{knownbound} implies the boundedness of the sequence $R(\eta(N))$.
\end{proof}

\begin{proposition}\label{propboundedseq2}
The sequence $\{H(\eta(N);s)\}_{N\geq 1}$ is unbounded for each $s\leq -1$. In fact, the sequences $\{H(\eta(N);-1)/\log(N+1)\}$ and $\{H(\eta(N);s)/N^{-(s+1)}\}$, $s<-1$, are bounded and divergent.
\end{proposition}
\begin{proof}
Suppose that $N$ has the binary decomposition \eqref{bindecomp}, and let 
\[
\theta_{k}=\frac{2^{n_{k}}}{N},\qquad 1\leq k\leq p.
\]
From \eqref{def:H} and \eqref{ineqtheta} we obtain that
\[
0<H(\eta(N);-1)=\sum_{k=1}^{p}\theta_{k}^{-1}\left(\theta_{k}-\sum_{j=k+1}^{p}\theta_{j}\right)\leq \sum_{k=1}^{p}\theta_{k}^{-1}\theta_{k}=p=\tau_{b}(N),
\]
see \eqref{def:arithfunc}. We also know by Theorem 1.1 in \cite{LopMc2} (see also formula (2.1) in \cite{LopMc2}) that
\[
\tau_{b}(N)\leq \frac{\log(N+1)}{\log 2},
\]
hence for all $N\in\mathbb{N}$ we have
\begin{equation}\label{ineqHminusone}
0<\frac{H(\eta(N);-1)}{\log(N+1)}\leq \frac{1}{\log 2}.
\end{equation}

Assume now that $s<-1$. Then by \eqref{def:H} and \eqref{ineqtheta} we obtain
\begin{align*}
|H(\eta(N);s)| & \leq \sum_{k=1}^{p}\theta_{k}^{s}\left(2(2^{s}+1)\left(\sum_{j=k+1}^{p}\theta_{j}\right)+\theta_{k}\right)\\
& \leq \sum_{k=1}^{p}\theta_{k}^{s}\,(2(2^{s}+1)\theta_{k}+\theta_{k})\\
&=(2^{s+1}+3)\sum_{k=1}^{p}\theta_{k}^{s+1}.
\end{align*}
We also have
\[
\sum_{k=1}^{p}\theta_{k}^{s+1}=N^{-(s+1)}\sum_{k=1}^{p}2^{n_{k}(s+1)}\leq N^{-(s+1)}\sum_{m=0}^{\infty}2^{(s+1)m}=\frac{N^{-(s+1)}}{1-2^{s+1}}.
\]
We conclude that for $s<-1$,
\[
|H(\eta(N);s)|\leq N^{-(s+1)}\frac{2^{s+1}+3}{1-2^{s+1}},
\]
so the sequence $H(\eta(N);s)/N^{-(s+1)}$ is bounded.

If we take
\begin{equation}\label{seqNp}
N=N(p)=\frac{4^{p}-1}{3}=\sum_{k=0}^{p-1}2^{2k},
\end{equation}
then 
\[
\sum_{j=k+1}^{p}\theta_{j}=\frac{1}{N}\frac{4^{p-k}-1}{3}
\]
and a simple calculation gives 
\[
H(\eta(N(p));-1)=\frac{2p}{3}+\frac{4}{9}(1-4^{-p}).
\]
Therefore
\[
\lim_{p\rightarrow\infty}\frac{H(\eta(N(p));-1)}{\log(N(p)+1)}=\frac{1}{3\log 2}.
\]
On the other hand, clearly
\[
\lim_{r\rightarrow\infty}\frac{H(\eta(2^{r});-1)}{\log(2^{r}+1)}=0.
\]
Hence the sequence $H(\eta(N);-1)/\log(N+1)$ is divergent.

For the sequence in \eqref{seqNp}, a careful calculation shows that for $s<-1$,
\[
\frac{H(\eta(N(p));s)}{N^{-(s+1)}}=\frac{1}{3}(2^{s+1}+1)\sum_{m=0}^{p-1}4^{(s+1)m}-\frac{2}{3}(2^{s}-1)\sum_{m=0}^{p-1}4^{sm}.
\]
Therefore
\[
\lim_{p\rightarrow\infty}\frac{H(\eta(N(p));s)}{N^{-(s+1)}}=\frac{2^{s+1}+1}{3(1-4^{s+1})}-\frac{2(2^{s}-1)}{3(1-4^{s})}=\frac{1+4^{s}(2^{s+1}-3)}{(1-4^{s+1})(1-4^{s})}>0.
\]
On the other hand,
\[
\lim_{r\rightarrow\infty}\frac{H(\eta(2^{r});s)}{2^{-r(s+1)}}=0,
\]
and so the sequence $H(\eta(N);s)/N^{-(s+1)}$ is divergent.
\end{proof}

\begin{theorem}\label{firstthm-expansion}
For every $s\geq -1$ such that $s\neq 0$ and $s\notin\mathbb{N}_{\mathrm{odd}}$, we have
\begin{equation}\label{energyexp}
E_{s}(\alpha_{N,s})=v(s) N^{2}+\frac{2}{(2\pi)^{s}}\sum_{j=0}^{\lfloor(s+1)/2\rfloor}\beta_{j}(s)\,\zeta(s-2j)\,H(\eta(N);s-2j)\,N^{s-2j+1}+\rho_{N,s},
\end{equation}
where the sequence $(\rho_{N,s})_{N=2}^{\infty}$ is uniformly bounded in $N$. In the special case that $s$ is a positive even integer, $v(s)=0$ and $\rho_{N,s}=0$ for all $N\geq 2$.
\end{theorem}
\begin{proof}
By \eqref{expLsN} we can write
\begin{align}
\mathcal{L}_{s}(2^{n_{k}}) & =v(s) 2^{2n_{k}}+\frac{2}{(2\pi)^{s}}\sum_{j=0}^{J}\beta_{j}(s)\zeta(s-2j) 2^{n_{k}(s-2j+1)}+O(2^{n_{k}(s-2J-1)}),\label{est1}\\
\mathcal{L}_{s}(2^{n_{k}+1}) & =v(s) 2^{2n_{k}+2}+\frac{2}{(2\pi)^{s}}\sum_{j=0}^{J}\beta_{j}(s)\zeta(s-2j) 2^{(n_{k}+1)(s-2j+1)}+O(2^{(n_{k}+1)(s-2J-1)}).\label{est2}
\end{align}
Observe also that if $N\in\mathbb{N}$ has the binary decomposition \eqref{bindecomp}, then
\begin{equation}\label{expNsquared}
N^{2}=\sum_{k=1}^{p-1}\left(\sum_{j=k+1}^{p} 2^{n_{j}-n_{k}}\right)2^{2(n_{k}+1)}+\sum_{k=1}^{p}\left(1-\sum_{j=k+1}^{p} 2^{n_{j}-n_{k}+1}\right) 2^{2n_{k}}.
\end{equation}
So applying \eqref{energybin}, \eqref{est1}, \eqref{est2}, and \eqref{expNsquared} we obtain
\begin{align*}
E_{s}(\alpha_{N,s})={} &\sum_{k=1}^{p-1}\left(\sum_{t=k+1}^{p} 2^{n_{t}-n_{k}}\right)\left(v(s) 2^{2(n_{k}+1)}+\frac{2}{(2\pi)^{s}}\sum_{j=0}^{J}\beta_{j}(s)\zeta(s-2j) 2^{(n_{k}+1)(s-2j+1)}\right)\\
&+\sum_{k=1}^{p}\left(1-\sum_{t=k+1}^{p} 2^{n_{t}-n_{k}+1}\right)\left(v(s) 2^{2 n_{k}}+\frac{2}{(2\pi)^{s}}\sum_{j=0}^{J}\beta_{j}(s)\zeta(s-2j) 2^{n_{k}(s-2j+1)}\right)\\
&+O\left(\sum_{k=1}^{p}2^{n_{k}(s-2J-1)}\right)\\
={} &v(s) N^{2}+\frac{2}{(2\pi)^{s}}\sum_{j=0}^{J}\beta_{j}(s)\zeta(s-2j)\sum_{k=1}^{p-1}\left(\sum_{t=k+1}^{p}2^{n_{t}-n_{k}}\right)2^{(n_{k}+1)(s-2j+1)}\\
&+\frac{2}{(2\pi)^{s}}\sum_{j=0}^{J}\beta_{j}(s)\zeta(s-2j)\sum_{k=1}^{p}\left(1-\sum_{t=k+1}^{p}2^{n_{t}-n_{k}+1}\right)2^{n_{k}(s-2j+1)}\\& +O\left(\sum_{k=1}^{p}2^{n_{k}(s-2J-1)}\right),
\end{align*}
where for the big-$O$ term we have used the simple estimates
\begin{equation}\label{bounds}
\sum_{t=k+1}^{p} 2^{n_{t}-n_{k}}\leq 1,\qquad |1-\sum_{t=k+1}^{p} 2^{n_{t}-n_{k}+1}|\leq 3.
\end{equation}
If we define
\[
\theta_{k}=\frac{2^{n_{k}}}{N},\qquad 1\leq k\leq p,
\]
then simple algebraic manipulations lead to the relation
\begin{align}
&\sum_{k=1}^{p-1}\left(\sum_{t=k+1}^{p} 2^{n_{t}-n_{k}}\right)2^{(n_{k}+1)(s-2j+1)}+\sum_{k=1}^{p}\left(1-\sum_{t=k+1}^{p}2^{n_{t}-n_{k}+1}\right)2^{n_{k}(s-2j+1)}\notag\\
&=\left(\sum_{k=1}^{p-1}\left(\sum_{t=k+1}^{p}\frac{2^{n_{t}+1}}{N}\right)\left(\frac{2^{n_{k}+1}}{N}\right)^{s-2j}+\sum_{k=1}^{p}\left(\frac{2^{n_{k}}}{N}-\sum_{t=k+1}^{p}\frac{2^{n_{t}+1}}{N}\right)\left(\frac{2^{n_{k}}}{N}\right)^{s-2j}\right)N^{s-2j+1}\notag\\
&=\left(\sum_{k=1}^{p-1}\left(\sum_{t=k+1}^{p} 2\theta_{t}\right)(2\theta_{k})^{s-2j}+\sum_{k=1}^{p}\left(\theta_{k}-\sum_{t=k+1}^{p} 2\theta_{t}\right)\theta_{k}^{s-2j}\right)N^{s-2j+1}\notag\\
&=H(\eta(N);s-2j) N^{s-2j+1}.\label{expansionHtheta}
\end{align}
So from the above calculations we conclude that
\[
E_{s}(\alpha_{N,s})=v(s)N^{2}+\frac{2}{(2\pi)^{s}}\sum_{j=0}^{J}\beta_{j}(s)\zeta(s-2j) H(\eta(N);s-2j) N^{s-2j+1}+O\left(\sum_{k=1}^{p} 2^{n_{k}(s-2J-1)}\right).
\]
If we take now $J=\lfloor\frac{s+1}{2}\rfloor$, then $s-2J-1<0$, and it follows that the sequence 
\[
N\mapsto \sum_{k=1}^{p} 2^{n_{k}(s-2J-1)}
\]
is bounded. This justifies \eqref{energyexp}.

In the case that $s$ is a positive even integer, taking $J=s/2$ the claim follows from \eqref{exactformLsN}. 
\end{proof}

\begin{remark}
In the case $s=-1$, since $\zeta(-1)=-1/12$ and $v(-1)=4/\pi$, we obtain from \eqref{energyexp} that
\begin{equation}\label{expan2}
E_{-1}(\alpha_{N,-1})=\frac{4}{\pi}N^{2}-\frac{\pi}{3} H(\eta(N);-1)+O(1).
\end{equation}
\end{remark}

To formulate the asymptotic expansion corresponding to $s\in \No$, we need some additional notation. For every integer $M\geq 0$, let us define the constant 
\begin{equation}\label{def:CM}
	C_M:=\frac{\beta'_M(2M+1)}{\beta_M(2M+1)}+\frac{1}{2}\psi(M+1)-\frac{1}{2}\psi(M+1/2),
\end{equation}
where $\psi=\Gamma'/\Gamma$ is the Digamma function. In this formula, $\beta_{M}'(2M+1)$ indicates the derivative of $\beta_{M}(s)$ at $s=2M+1$, and according to Theorem 1.3 in \cite{BHS}, we have
\[
\beta_{M}'(2M+1)=\sum_{k=0}^{M-1}\beta_{k}(2M+1)\frac{\zeta(2(M-k))}{M-k}.
\] 
In particular, since $\beta_0(s)=1$, $\psi(1)=-\gamma$, and $\psi(1/2)=-\gamma-2\log 2$, we have
\[
C_0=\log 2.
\]

\begin{theorem}\label{secondthm-expansion}
If $s\in\No$, with $s=2M+1$,
\begin{align}\label{expansionoddcase}
		\begin{split}
			E_{s}(\alpha_{N,s})= {} &
			\frac{(\frac{1}{2})_M}{\pi\,2^{2M}M!}N^2\log N+	\frac{(\frac{1}{2})_M}{\pi\,2^{2M}M!}(\gamma-\log(\pi) +C_M+K(\eta(N)))N^2
			\\
			&+\frac{2}{(2\pi)^s}\sum_{j=0}^{M-1}\beta_j(s)\,\zeta(s-2j)\,H(\eta(N);s-2j)\,N^{s-2j+1}+\rho_{N,s},
		\end{split}
	\end{align}
	where the sequence $(\rho_{N,s})_{N=2}^\infty$ is uniformly bounded in $N$. 
\end{theorem}
\begin{proof}
Suppose that $s=2M+1$ is a positive odd integer. In this proof we use the notation
\[
q_{M}:=\frac{(\frac{1}{2})_M}{\pi\,2^{2M}M!}.
\]
By \cite[Thm. 1.3]{BHS} we know that for any $J>M$ the following expansion holds as $N\rightarrow\infty$:
\begin{align*}
\mathcal{L}_{s}(N)={} &q_{M} N^{2}\log N+q_{M}(\gamma-\log(\pi)+C_{M})N^{2}\\
&+\frac{2}{(2\pi)^{s}}\underset{j\not=M}{\sum_{j=0}^{J}}\beta_{j}(s)\zeta(s-2j) N^{s-2j+1}+O(N^{s-2J-1}).
\end{align*}
So we have
\begin{align*}
\mathcal{L}_{s}(2^{n_{k}})={} &q_{M}\,2^{2n_{k}}\log(2^{n_{k}})+q_{M}(\gamma-\log(\pi)+C_{M})\,2^{2n_{k}}\\
&+\frac{2}{(2\pi)^{s}}\underset{j\not=M}{\sum_{j=0}^{J}}\beta_{j}(s)\zeta(s-2j) 2^{n_{k}(s-2j+1)}+O(2^{n_{k}(s-2J-1)})
\end{align*}
and
\begin{align*}
\mathcal{L}_{s}(2^{n_{k}+1})={} &q_{M}\,2^{2n_{k}+2}\log(2^{n_{k}+1})+q_{M}(\gamma-\log(\pi)+C_{M})\,2^{2n_{k}+2}\\
&+\frac{2}{(2\pi)^{s}}\underset{j\not=M}{\sum_{j=0}^{J}}\beta_{j}(s)\zeta(s-2j) 2^{(n_{k}+1)(s-2j+1)}+O(2^{(n_{k}+1)(s-2J-1)}).
\end{align*}
It easily follows from \eqref{energybin}, \eqref{expNsquared}, \eqref{bounds}, \eqref{expansionHtheta}, and the two previous identities that
\begin{align}
E_{s}(\alpha_{N,s})={} &q_{M}(\gamma-\log(\pi)+C_{M})N^{2}+q_{M} S_{N}\notag\\
&+\frac{2}{(2\pi)^s}\underset{j\not=M}{\sum_{j=0}^{J}}\beta_{j}(s)\zeta(s-2j)H(\eta(N);s-2j)N^{s-2j+1}+O\left(\sum_{k=1}^{p} 2^{n_{k}(s-2J-1)}\right),\label{expenergia}
\end{align}
where
\[
S_{N}:=\sum_{k=1}^{p-1}\left(\sum_{t=k+1}^{p}2^{n_{t}-n_{k}}\right)2^{2n_{k}+2}\log(2^{n_{k}+1})+\sum_{k=1}^{p}\left(1-\sum_{t=k+1}^{p} 2^{n_{t}-n_{k}+1}\right)2^{2n_{k}}\log(2^{n_{k}}).
\]
Using \eqref{expNsquared} we can write
\begin{align*}
N^{2}\log N={} &\sum_{k=1}^{p-1}\left(\sum_{t=k+1}^{p}2^{n_{t}-n_{k}}\right)\left(2^{2n_{k}+2}\log(2^{n_{k}+1})+2^{2n_{k}+2}\log\left(\frac{N}{2^{n_{k}+1}}\right)\right)\\
&+\sum_{k=1}^{p}\left(1-\sum_{t=k+1}^{p}2^{n_{t}-n_{k}+1}\right)\left(2^{2 n_{k}}\log(2^{n_{k}})+2^{2 n_{k}}\log\left(\frac{N}{2^{n_{k}}}\right)\right),
\end{align*}
whence it follows that
\begin{align*}
S_{N}={} &N^{2}\log N+\sum_{k=1}^{p-1}\left(\sum_{t=k+1}^{p}2^{n_{t}-n_{k}}\right)2^{2n_{k}+2}\log\left(\frac{2^{n_{k}+1}}{N}\right)\\
&+\sum_{k=1}^{p}\left(1-\sum_{t=k+1}^{p}2^{n_{t}-n_{k}+1}\right)2^{2 n_{k}}\log\left(\frac{2^{n_{k}}}{N}\right)\\
={} &N^{2}\log N+N^{2}\left(4\sum_{k=1}^{p-1}\left(\sum_{t=k+1}^{p}\theta_{t}\,\theta_{k}\right)\log(2\theta_{k})+\sum_{k=1}^{p}\left(\theta_{k}^{2}-\sum_{t=k+1}^{p}2\theta_{t}\,\theta_{k}\right)\log(\theta_{k})\right),
\end{align*}
where $\theta_{k}=2^{n_{k}}/N$, $1\leq k\leq p$. 

Applying the identity
\[
1=(\theta_{1}+\cdots+\theta_{p})^2=\sum_{k=1}^{p}\theta_{k}^2+2\sum_{k=1}^{p-1}\sum_{t=k+1}^{p}\theta_{t}\,\theta_{k},
\]
it is easy to check that
\[
K(\eta(N))=4\sum_{k=1}^{p-1}\left(\sum_{t=k+1}^{p}\theta_{t}\,\theta_{k}\right)\log(2\theta_{k})+\sum_{k=1}^{p}\left(\theta_{k}^{2}-\sum_{t=k+1}^{p}2\theta_{t}\,\theta_{k}\right)\log(\theta_{k}).
\]
Therefore
\[
S_{N}=N^{2}\log N+N^{2} K(\eta(N)).
\]
From this identity and \eqref{expenergia} we conclude that
\begin{align}
E_{s}(\alpha_{N,s})={}& q_{M} N^{2}\log N+q_{M}(\gamma-\log(\pi)+C_{M}+K(\eta(N)))N^{2}\notag\\
& +\frac{2}{(2\pi)^s}\underset{j\not=M}{\sum_{j=0}^{J}}\beta_{j}(s)\zeta(s-2j)H(\eta(N);s-2j)N^{s-2j+1}+O\left(\sum_{k=1}^{p} 2^{n_{k}(s-2J-1)}\right).\label{expenergia2}
\end{align} 
Finally we take $J=M+1$ in \eqref{expenergia2}. Since $s-2J+1=0$, the sequence 
\[
\frac{2}{(2\pi)^{s}}\beta_{M+1}(s)\zeta(s-2J) H(\eta(N);s-2J) N^{s-2J+1}=O(1)
\]
is bounded. Also, with $J=M+1$ we have $s-2J-1=-2$, so the sequence
\[
N\mapsto\sum_{k=1}^{p}2^{n_{k}(s-2J-1)}
\]
is also bounded, and \eqref{expansionoddcase} follows.
\end{proof}

\noindent\textbf{Proof of Theorem \ref{theo:leadasympTNs}:} Recall that $\beta_{0}(s)=1$. The identity \eqref{expTN-1} follows from \eqref{expan2}. In the range $-1<s<1$, $s\neq 0$, according to \eqref{energyexp} we have
\[
E_{s}(\alpha_{N,s})=v(s) N^{2}+\frac{2\zeta(s)}{(2\pi)^{s}}H(\eta(N);s)N^{1+s}+O(1),
\]
and so
\[
\frac{E_{s}(\alpha_{N,s})-v(s) N^{2}}{N^{1+s}}=\frac{2\zeta(s)}{(2\pi)^{s}}H(\eta(N);s)+O(N^{-1-s}).
\]
If $1<s<3$ and $s\neq 2$, then $\lfloor\frac{s+1}{2}\rfloor=1$ and $0<s-1<2$, and from \eqref{energyexp} we obtain
\begin{equation}\label{expansion}
E_{s}(\alpha_{N,s})=v(s) N^{2}+\frac{2\zeta(s)}{(2\pi)^{s}}H(\eta(N);s)N^{1+s}+\frac{2\zeta(s-2)}{(2\pi)^{s}}\beta_{1}(s)H(\eta(N);s-2)N^{s-1}+O(1),
\end{equation}
which, taking into account that the sequence $H(\eta(N);s-2)$ is bounded, gives
\[
\frac{E_{s}(\alpha_{N,s})}{N^{s+1}}=\frac{2\zeta(s)}{(2\pi)^{s}}H(\eta(N);s)+O(N^{1-s}).
\]
If $s=2$, then \eqref{expansion} is valid with $v(s)=0$ and $O(1)=0$, hence in this case
\begin{equation}\label{expansion2}
\frac{E_{s}(\alpha_{N,s})}{N^{s+1}}=\frac{2\zeta(s)}{(2\pi)^{s}}H(\eta(N);s)+O(N^{-2}).
\end{equation}
If $s>3$, $s\notin\mathbb{N}_{\mathrm{odd}}$, then the expansion \eqref{expansion} is valid with $O(1)$ replaced by $O(N^{s-3})$ (here we use Proposition \ref{propboundedseq}), and since $s-1>2$, it follows that \eqref{expansion2} also holds in this case.

In the case $s=1$, by \eqref{expansionoddcase} we know that
\[
E_{1}(\alpha_{N,1})=\frac{1}{\pi} N^{2}\log N+\frac{1}{\pi}(\gamma+\log(2/\pi)+K(\eta(N)))N^{2}+O(1),
\]
and \eqref{expTN1} follows.

According to \eqref{expansionoddcase}, if $s=3$ we have
\[
E_{s}(\alpha_{N,s})=\frac{2\zeta(s)}{(2\pi)^{s}}H(\eta(N);s)N^{4}+\frac{1}{8\pi} N^{2}\log N+\frac{1}{8\pi}(\gamma-\log(\pi)+C_{1}+K(\eta(N)))N^{2}+O(1),
\]
from which the last case in \eqref{expTNs} follows.

If $s>3$, $s\in\mathbb{N}_{\mathrm{odd}}$, then by \eqref{expansionoddcase} we have
\[
E_{s}(\alpha_{N,s})=\frac{2\zeta(s)}{(2\pi)^{s}}H(\eta(N);s)N^{s+1}+O(N^{s-1}),
\]
so \eqref{expansion2} holds in this case.

In formula (46) in \cite{LopWag} it was proved that if $N$ has the binary decomposition \eqref{bindecomp}, then
\[
T_{N,0}=-\log(2)\sum_{k=1}^{p}(n_{k}+2(k-1))\frac{2^{n_{k}}}{N}+\log N.
\]
Writing $\log(2^{n_{k}})=n_{k}\log 2$ we obtain
\[
-\log(2)\sum_{k=1}^{p}n_{k}\frac{2^{n_{k}}}{N}+\log N=-\sum_{k=1}^{p}\frac{2^{n_{k}}}{N}\log\frac{2^{n_{k}}}{N},
\]
so if we use the notation $\eta(N)=(\theta_{1},\ldots,\theta_{p})=(2^{n_{1}}/N,\ldots,2^{n_{p}}/N)$, then
\[
T_{N,0}=-\sum_{k=1}^{p}\theta_{k}\log\theta_{k}-\sum_{k=1}^{p}\log(4)(k-1)\theta_{k}=R(\eta(N)).
\]
This concludes the proof.
\qed

\bigskip

\noindent\textbf{Proof of Corollary \ref{cor1}:}
If $s=-1$, then according to \eqref{expTN-1} and \eqref{doubperiod}, as $N\rightarrow\infty$ we have
\begin{align*}
T_{2N,-1}-T_{N,-1} & =-\frac{\pi}{3}\frac{H(\eta(2N);-1)}{\log(2N)}+\frac{\pi}{3}\frac{H(\eta(N);-1)}{\log N}+o(1)\\
& =\frac{\pi}{3}\frac{\log(2) H(\eta(N);-1)}{\log(N)\log(2N)}+o(1)
\end{align*}
and the claim follows now from Proposition \ref{propboundedseq2}. In the case $s>-1$, the result follows immediately from \eqref{doubperiod} and Theorem \ref{theo:leadasympTNs}.
\qed

\section{On the limit points of the sequences $H(\eta(N);s)$, $K(\eta(N))$, and $R(\eta(N))$}\label{limit-points-functions-H-K-R}
 	
Recall that $p=\tau_{b}(N)$ denotes the number of terms in the binary expansion of $N\in\mathbb{N}$. For each such $N$, we identify the vector
\[
\eta(N)=\l(\frac{2^{n_1}}{N},\ldots,\frac{2^{n_p}}{N}\r)
\] with the infinite vector
\[
\eta(N)=(\theta_{1},\theta_{2},\ldots,\theta_{n},\ldots)
\]
where $\theta_k=2^{n_k}/N$ for $1\leq k\leq p$, and $\theta_k=0$ for $k>p$.

Trivially, we have 
\[
\sum_{j=1}^{\infty}\theta_{j}=1.
\]

\begin{proposition}\label{propineqs}
Let $N\in\mathbb{N}$ and $\eta(N)=(\theta_{1},\theta_{2},\ldots,\theta_{n},\ldots)$. Then,
\begin{align*}
	1/2<\theta_{1}\leq 1, 
\end{align*}	
and  for each $ k\geq 1$, we have
\begin{align}
\theta_{k} & \leq \frac{1}{2^{k}-1},\label{firstineq}\\
\sum_{j=k+1}^{\infty}\theta_{j}& \leq \theta_{k},\label{thirdineq}\\
\sum_{j=1}^{k}\theta_{j} & \geq 1-\frac{1}{2^{k-1}}.\nonumber
\end{align}
\end{proposition}
\begin{proof}
\eqref{firstineq} is trivial for $k> p=\tau_b(N)$. As for $1\leq k\leq p$, writing $N$ as in \eqref{bindecomp} we obtain 
\[
\theta_{k}=\frac{2^{n_{k}}}{N}=\frac{1}{\sum_{i=1}^{p}2^{n_{i}-n_{k}}}\leq \frac{1}{\sum_{i=1}^{k}2^{n_{i}-n_{k}}}\leq \frac{1}{\sum_{i=1}^{k}2^{k-i}}=\frac{1}{2^{k}-1}.
\]
In the case $k=1$, we additionally  have the lower bound 
\[
\theta_{1}=\frac{1}{\sum_{i=1}^{p}2^{n_{i}-n_{1}}}>\frac{1}{2}.
\]
The inequality \eqref{thirdineq} follows from \eqref{ineqtheta}. Since $\sum_{j=1}^{\infty}\theta_{j}=1$ and $\theta_{j}\leq 2^{-j+1}$, we have
\[
\sum_{j=1}^{k}\theta_{j}=1-\sum_{j=k+1}^{\infty}\theta_{j}\geq 1-\sum_{j=k+1}^{\infty}2^{-j+1}=1-2^{-k+1}.
\] 
\end{proof}

Consider now the set of real-valued sequences 
\[
X:=\{\vec{x}=(x_1,x_2,\ldots,x_n,\ldots)\in \R^\N: |x_k|\leq 2^{-k+1}, \ k\geq 1\}.
\]
This is a metric subspace of $\ell^\infty$
with the metric induced by the $\ell^\infty$-norm
\[
\|\vec{x}-\vec{y}\|_{\ell^\infty}=\sup_{k\geq 1}|x_{k}-y_{k}|.
\] 
One can easily verify that a sequence $\vec{x}_N=(x_{1,N},\ldots,x_{n,N},\ldots)$, $N\geq 1$, of elements in $X$ converges to an element $\vec{x}=(x_1,x_2,\ldots,x_n,\ldots)\in \ell^\infty$ if and only if it does it in a pointwise sense:
\[
\lim_{N\to\infty}x_{k,N}=x_k, \qquad k\geq 1,
\] 
in which case the limit $\vec{x}\in X$. The set $X$ is closed, as it is clear that the limit $\vec{x}$ of a sequence $(\vec{x}_N)\subset X$ is also a member of $X$. Moreover, since every bounded sequence of elements in $\ell^\infty$ has a pointwise convergent subsequence, the space $X$ is compact.

In view of \eqref{firstineq} (and since $(2^{k}-1)^{-1}\leq 2^{-(k-1)}$ for all $k\geq 1$), we have  
\begin{equation}\label{defSeta}
\mathcal{S}_\eta:=\{\eta(N): N\in \N\}\subset X.
\end{equation}
We define $\mathcal{S}:=\cj{\mathcal{S}_\eta}$
to be the closure of $\mathcal{S}_\eta$ in $X$. Since $X$ is compact, so is $\mathcal{S}$. 

Due to the periodicity property that $\eta(2N)=\eta(N)$, every element $\eta(N)\in \mathcal{S}_\eta$ can be obtained as the limit of elements $\eta(N_k)\in \mathcal{S}_\eta$ for some subsequence $(N_k)$ of the natural numbers (say e.g.  $N_k=2^kN$, $k\geq 1$). Hence, $\mathcal{S}$ coincides with the set  of all sequences  
\[
(\vartheta_{1},\vartheta_{2}, \vartheta_{3},\ldots,\vartheta_{n},\ldots)
\]
such that  there exists a subsequence  $\mathcal{N}$ of  the natural numbers    such that 
\begin{equation}\label{17-7-2}
\vartheta_{j}=\lim_{N\in \mathcal{N}}\theta_{j,N},\qquad j\geq 1,
\end{equation}
where
\begin{equation}\label{20-7-1}
\eta(N)=(\theta_{1,N},\theta_{2,N},\ldots,\theta_{n,N},\ldots),\qquad N\in  \mathcal{N}.
\end{equation}

For example, if we take the sequence $N_{k}=2^{k}-1=2^{k-1}+2^{k-2}+\cdots+1$, $k\in \N$, then 
\[
\lim_{k\rightarrow\infty}\theta_{j,N_k}=\lim_{k\rightarrow\infty}\frac{2^{k-j}}{2^{k}-1}=2^{-j},
\] 
hence the vector 
\[
\left(\frac{1}{2},\frac{1}{2^{2}},\frac{1}{2^{3}},\ldots,\frac{1}{2^{n}},\ldots\right)\in\mathcal{S}.
\]
Another example is the vector
\[
\left(\frac{3}{4},\frac{3}{4^{2}},\frac{3}{4^{3}},\ldots,\frac{3}{4^{n}},\ldots\right)\in\mathcal{S},
\] 
obtained by taking $N_{k}=\frac{4^{k}-1}{3}=4^{k-1}+4^{k-2}+\cdots+1$.

\begin{proposition}\label{propertiesvarthetavectors}
If $(\vartheta_{1},\vartheta_{2},\vartheta_{3},\ldots)\in\mathcal{S}$, then	
\begin{align}
	1/2\leq \vartheta_{1}\leq{} & 1,\label{varthetaprop1}\\
	\sum_{k=1}^{\infty}\vartheta_{k}={} &1,\label{varthetaprop2}\\
		\sum_{k=1}^{\infty}\vartheta_{k}\,|\log \vartheta_{k}|\leq{} & \log 4,\label{varthetaprop6}
\end{align}
and  for each $ k\geq 1$, we  have
\begin{align}
0\leq 	\vartheta_{k} \leq{} & \frac{1}{2^{k}-1},\label{varthetaprop3}\\
	\sum_{j=k+1}^{\infty}\vartheta_{j} \leq{} & \vartheta_{k},\label{varthetaprop5}\\
	\sum_{j=1}^{k}\vartheta_{j}  \geq {} &1-\frac{1}{2^{k-1}}.\label{varthetaprop4}
\end{align}
\end{proposition}
\begin{proof} Let $\mathcal{N}$ be a subsequence of $\N$ such that (with the notation \eqref{20-7-1})
	\[
	\vartheta_{j}=\lim_{N\in\mathcal{N}}\theta_{j,N},\qquad j\geq 1.
	\] 
	Since $\theta_{k,N}\leq 2^{-k+1}$ for all $k\geq 1$, we have 
\begin{equation}\label{Lebdom1}
	\sum_{k=1}^{\infty}\theta_{k,N}\leq\sum_{k=1}^{\infty}2^{-(k-1)}<\infty.
\end{equation}
The function $x\log x$ is decreasing on the interval $[0,e^{-1}]$, hence for every $k\geq 3$ the inequalities $\theta_{k,N}\leq 2^{-k+1}<e^{-1}$ imply that $\theta_{k,N}|\log \theta_{k,N}|\leq (k-1)\log(2)\,2^{-k+1}$, thus	
\begin{equation}\label{Lebdom2}
\sum_{k=1}^{\infty}\theta_{k,N}|\log\theta_{k,N}|\leq \frac{2}{e}+\log (2)\sum_{k=3}^{\infty}(k-1)2^{-(k-1)}<\infty.
\end{equation}
By an application of Lebesgue's dominated convergence theorem, it follows from \eqref{Lebdom1}, \eqref{Lebdom2}, and \eqref{knownbound} that 
	\[
	\sum_{k=1}^{\infty}\vartheta_{k}=\lim_{N\in \mathcal{N}}\sum_{k=1}^{\infty}\theta_{k,N}=1,\qquad\sum_{k=1}^{\infty}\vartheta_{k}|\log\vartheta_k|=\lim_{N\in \mathcal{N}}\sum_{k=1}^{\infty}\theta_{k,N}|\log\theta_{k,N}|\leq \log 4.
	\] 
This establishes \eqref{varthetaprop2} and \eqref{varthetaprop6}. The remaining relations \eqref{varthetaprop1},   \eqref{varthetaprop3},  \eqref{varthetaprop5}, and  \eqref{varthetaprop4} follow directly from Proposition \ref{propineqs} by passage to the limit (in deriving  \eqref{varthetaprop5} we can simply write 	$	\sum_{j=k+1}^{\infty}\theta_{j,N}=1-\sum_{j=1}^{k}\theta_{j,N}\to 1-\sum_{j=1}^{k}\vartheta_{j}=\sum_{j=k+1}^{\infty}\vartheta_{j}$).
\end{proof}

\begin{definition}\label{def:extensionofHKR}
For a vector $\vec{\vartheta}=(\vartheta_{1}, \vartheta_{2},\ldots,\vartheta_{n},\ldots)\in\mathcal{S}$ and $s>-1$, we define
\begin{align}
H(\vec{\vartheta};s) & :=\sum_{n=1}^{\infty}\vartheta_{n}^{s+1}+2(2^{s}-1)\sum_{n=1}^{\infty}\sum_{j=n+1}^{\infty}\vartheta_{n}^{s}\,\vartheta_{j},\label{19-7-1}\\
K(\vec{\vartheta}) & :=2\log 2+\sum_{n=1}^{\infty}\vartheta_{n}^{2}\,\log(\vartheta_{n}/4)+2\sum_{n=1}^{\infty}\sum_{j=n+1}^{\infty}\vartheta_{j}\,\vartheta_{n}\log\vartheta_{n},\label{19-7-2}\\
R(\vec{\vartheta}) & :=-(2\log 2)\sum_{n=1}^{\infty} (n-1)\vartheta_{n}-\sum_{n=1}^{\infty}\vartheta_{n}\log \vartheta_{n}.\label{19-7-3}
\end{align}
Naturally, in these expressions we understand that if $\vartheta_{n}=0$, then $\vartheta_{n}\log\vartheta_{n}=0$ and $\vartheta_{n}^{s}\,\vartheta_{j}=0$ for $-1<s<0$ and $j\geq n+1$.
\end{definition}
\begin{remark}
Observe that for all $N\in\N$, the values $H(\eta(N);s)$, $K(\eta(N))$, and $R(\eta(N))$ as defined by \eqref{19-7-1}--\eqref{19-7-3} coincide with those previously defined by \eqref{def:H}, \eqref{def:K}, and \eqref{def:Phi}. 
\end{remark}

\begin{proposition}\label{prop_contHKR}
The functions given by \eqref{19-7-1}--\eqref{19-7-3} are well-defined for each $\vec{\vartheta}\in\mathcal{S}$, and each of the functions  $H(\vec{\vartheta};s)$, $s>-1$, $K(\vec{\vartheta})$, and $R(\vec{\vartheta})$ is continuous on $\mathcal{S}$.
\end{proposition}
\begin{proof}
Let $\vec{\vartheta}_N=(\vartheta_{1,N}, \vartheta_{2,N},\dots)$, $N\geq 1$, be a sequence in $\mathcal{S}$ converging to $\vec{\vartheta}=(\vartheta_{1},\vartheta_{2},\ldots)\in\mathcal{S}$.  By Proposition \ref{propertiesvarthetavectors}, we have that for all $n,N\geq 1$, 
\[
\sum_{j=n+1}^{\infty}\vartheta_{j,N}\leq \vartheta_{n,N}\leq \frac{1}{2^{n-1}},
\]
so that 
\[
\left|\vartheta_{n,N}^{s+1}+2(2^s-1)\vartheta_{n,N}^{s}\sum_{j=n+1}^{\infty}\vartheta_{j,N}\right|\leq (2^{s+1}+3)2^{-(n-1)(s+1)} 
\]
(note that $2^{s}-1<0$ for $-1<s<0$). This implies that $H(\cdot;s)$ is well-defined on $\mathcal{S}$, and moreover, since for all $n\geq 1$, we have
\[
\lim_{N\to\infty}\left(\vartheta_{n,N}^{s+1}+2(2^s-1)\vartheta_{n,N}^{s}\sum_{j=n+1}^{\infty}\vartheta_{j,N}\right)=\vartheta_{n}^{s+1}+2(2^s-1)\vartheta_{n}^{s}\sum_{j=n+1}^{\infty}\vartheta_{j}
\]
(recall that $\sum_{j=n+1}^{\infty}\vartheta_{j,N}=1-\sum_{j=1}^n\vartheta_{j,N})$, it also implies by the Lebesgue dominated convergence theorem that 
\[
\lim_{N\to\infty}H(\vec{\vartheta}_N;s)=H(\vec{\vartheta};s),
\]
proving the continuity of $H(\cdot;s)$ on $\mathcal{S}$.

The function $x\log x$ is $\leq 0$ on $[0,1]$ and attains its minimum at $x=e^{-1}$.  Hence we have  $\vartheta_{n,N}\,|\log \vartheta_{n,N}|\leq e^{-1}$ for all $n\geq 1$, and so  
 \[
 \left|\vartheta_{n,N}^{2}\log(\vartheta_{n,N}/4)+2\vartheta_{n,N}\log\vartheta_{n,N}\sum_{j=n+1}^{\infty}\vartheta_{j,N}\right|\leq 6e^{-1}\vartheta_{n,N}\leq   6e^{-1}2^{1-n}.
 \]
Note also that the function $x\log x$ is decreasing on the interval $[0,e^{-1}]$ and $\vartheta_{n,N}\leq (2^{n}-1)^{-1}<e^{-1}$ for all $n\geq 2$, so that $\vartheta_{n,N}\,|\log \vartheta_{n,N}|\leq 2^{-n+1}\log(2^{n}-1)$ for all $n\geq 2$, and thus
 \[
|(2\log(2) (n-1)+\log \vartheta_{n,N})\,\vartheta_{n,N}|\leq 2\log(2)(n-1)2^{-n+1}+2^{-n+1}\log(2^{n}-1),\quad n\geq 2.
\]
The last two inequalities show that the functions $K(\cdot)$ and  $R(\cdot)$  are well-defined and continuous on $\mathcal{S}$.
\end{proof}

\begin{definition}
We say that $x\in\mathbb{R}$ is a limit point of a bounded sequence $(x_{n})_{n=1}^{\infty}$ of real numbers if there is a subsequence of $(x_{n})_{n=1}^{\infty}$ that converges to $x$.
\end{definition}

The following result is a direct consequence of the continuity of the functions $H(\cdot;s)$, $K(\cdot)$, and  $R(\cdot)$ on the set $\mathcal{S}$, and the fact that $\mathcal{S}$ is compact.

\begin{theorem}\label{theo:limpointHKR}
For each $s>-1$, the set $\{H(\vec{\vartheta};s): \vec{\vartheta}\in\mathcal{S}\}$ is the set of all limit points of the sequence $(H(\eta(N);s))_{N}$. Likewise, $\{K(\vec{\vartheta}): \vec{\vartheta}\in\mathcal{S}\}$ and $\{R(\vec{\vartheta}): \vec{\vartheta}\in\mathcal{S}\}$ are the sets of limit points of the sequences $(K(\eta(N)))_{N}$ and $(R(\eta(N)))_{N}$, respectively.
\end{theorem}

\section{Characterizing the set $\mathcal{S}$}\label{characterization-set-S}

\begin{proposition}\label{prop_19-7-4}
	If $(\vartheta_{1},\vartheta_{2},\vartheta_{3},\ldots)\in \mathcal{S}$, and $\vartheta_{k}>0$ for some $k\geq 2$, then
	\[
	\vartheta_{k}=\frac{\vartheta_{k-1}}{2^{r}}
	\]
	for some integer $r\geq 1$ that may depend on $k$.
\end{proposition}
\begin{proof}
	Let $\mathcal{N}\subset\N$ be a sequence of integers $N$ such that 
	\[
	\vartheta_{j}=\lim_{N\in\mathcal{N}}\theta_{j,N},\qquad j\geq 1.
	\] 
	Since $\vartheta_{k}>0$ and $\theta_{k,N}\to \vartheta_{k}$ as $N\to\infty$, there must be a natural number $M$ such that $\theta_{k,N}>0$ for all $N\in\mathcal{N}$ such that $N>M$. This implies that $\theta_{j,N}=2^{n_j}/N$ for all $1\leq j\leq k$ and $N>M$. Thus we have
	\[
	\frac{\vartheta_{k-1}}{\vartheta_{k}}=\lim_{N\in\mathcal{N}}\frac{2^{n_{k-1}}}{N}\frac{N}{2^{n_{k}}}=\lim_{N\in\mathcal{N}}2^{n_{k-1}-n_{k}}.
	\]
	Therefore $2^{n_{k-1}-n_{k}}\in\N$ is eventually constant and the claim follows. 
\end{proof}

\begin{remark}An immediate consequence of Proposition \ref{prop_19-7-4} is that if
	$(\vartheta_{1},\vartheta_{2},\vartheta_{3},\ldots)\in\mathcal{S}$, then there exists $m\in \N\cup\{\infty\}$, $m\geq 2$,  such that $\vartheta_{1}>\vartheta_{2}>\cdots>\vartheta_{m-1}>0$ and  $\vartheta_{k}=0$ for all $k\geq m$.  
\end{remark}

The following is a characterization of $\mathcal{S}$.

\begin{proposition}\label{prop_characS}
Let $\vec{\vartheta}=(\vartheta_{1}, \vartheta_{2},\ldots)$ be an infinite vector of non-negative real numbers such that $\vartheta_{j}\geq \vartheta_{j+1}$ for all $j\geq 1$. Then $\vec{\vartheta}\in\mathcal{S}$ if and only if $1/2\leq \vartheta_{1}\leq 1$ and for each $j$ such that $\vartheta_{j}>0$ we have $\vartheta_{j}=\vartheta_{1} 2^{-k_{j}}$, where $0=k_{1}<k_{2}<k_{3}<\cdots$ is an increasing sequence of non-negative integers such that
\begin{equation}\label{7-8-1}
\frac{1}{\vartheta_{1}}=1+\frac{1}{2^{k_{2}}}+\frac{1}{2^{k_{3}}}+\cdots
\end{equation}
is a binary expansion of $1/\vartheta_{1}$ (of finite or infinite length).
\end{proposition}
\begin{proof}
Suppose first that $\vec{\vartheta}\in\mathcal{S}$. Then, it follows from Proposition \ref{propineqs} and \eqref{17-7-2} that $1/2\leq \vartheta_{1}\leq 1$. By Proposition \ref{prop_19-7-4}, if $\vartheta_{j}>0$, then $\vartheta_{j}=\vartheta_{1} 2^{-k_{j}}$ for some non-negative integer $k_{j}$, and we have $0=k_{1}<k_{2}<k_{3}<\cdots$. Since $\sum_{j=1}^{\infty}\vartheta_{j}=1$, \eqref{7-8-1} follows.

Suppose now that the vector $\vec{\vartheta}=(\vartheta_{1},\vartheta_{2},\ldots)$ satisfies $1/2\leq \vartheta_{1}\leq 1$, and $\vartheta_{j}=\vartheta_{1} 2^{-k_{j}}$ for each positive $\vartheta_{j}$, where $0=k_{1}<k_{2}<\cdots$ is an increasing sequence of integers such that \eqref{7-8-1} holds. If $\vec{\vartheta}$ has only finitely many positive components, say $m$ of them, then
\[
\frac{1}{\vartheta_{1}}=1+\frac{1}{2^{k_{2}}}+\frac{1}{2^{k_{3}}}+\cdots+\frac{1}{2^{k_{m}}}=\frac{2^{k_{m}}+2^{k_{m}-k_{2}}+2^{k_{m}-k_{3}}+\cdots+1}{2^{k_{m}}},
\] 
and so if we take $N=2^{k_{m}}+2^{k_{m}-k_{2}}+2^{k_{m}-k_{3}}+\cdots+1$, then
\[
(\vartheta_{1},\ldots,\vartheta_{m})=\left(\vartheta_{1},\frac{\vartheta_{1}}{2^{k_{2}}},\ldots,\frac{\vartheta_{1}}{2^{k_{m}}}\right)=\eta(N).
\] 
Therefore, $\vec{\vartheta}\in \mathcal{S}_{\eta}$ (see \eqref{defSeta}), and so $\vec{\vartheta}\in \mathcal{S}$. If, on the other hand, every component of $\vec{\vartheta}$ is positive, then
\begin{equation}\label{7-8-2}
\frac{1}{\vartheta_{1}}=\sum_{j=1}^{\infty}\frac{1}{2^{k_{j}}}.
\end{equation}
If we take the partial sums
\[
\sum_{j=1}^{l}\frac{1}{2^{k_{j}}}=\frac{2^{k_{l}}+2^{k_{l}-k_{2}}+2^{k_{l}-k_{3}}+\cdots+1}{2^{k_{l}}}
\]
and define the integers $N_{l}=2^{k_{l}}+2^{k_{l}-k_{2}}+2^{k_{l}-k_{3}}+\cdots+1$, then by \eqref{7-8-2} we obtain 
\[
\lim_{l\rightarrow\infty}\frac{2^{k_{l}}}{N_{l}}=\vartheta_{1},
\]   
therefore 
\[
\lim_{l\rightarrow\infty}\frac{2^{k_{l}-k_{m}}}{N_{l}}=\frac{\vartheta_{1}}{2^{k_{m}}}=\vartheta_{m},\qquad m\geq 2. 
\]
This implies that $\vec{\vartheta}\in\mathcal{S}$.
\end{proof}

\begin{remark}
Proposition \ref{prop_characS} asserts that the vectors in $\mathcal{S}$ are precisely those that can be constructed by the following procedure. Start with a number $1/2\leq x\leq 1$, take a binary expansion of $1/x$ of the form
\[
\frac{1}{x}=\frac{1}{2^{k_{1}}}+\frac{1}{2^{k_{2}}}+\cdots, 
\]
where $0=k_{1}<k_{2}<\cdots$. If the expansion has $m$ terms, construct the vector
\[
\left(x,\frac{x}{2^{k_{2}}},\ldots,\frac{x}{2^{k_{m}}},0,0,\ldots,0,\ldots\right),
\] 
and if the expansion has infinitely many terms, construct the vector
\[
\left(x,\frac{x}{2^{k_{2}}},\ldots,\frac{x}{2^{k_{j}}},\ldots\right).
\]
\end{remark}

\begin{remark}\label{rmk_binexp}
We leave to the reader to check that a positive number $x>0$ has a non-unique binary expansion if and only if $x$ is an integer, or $x$ can be written in the form $q/2^{k}$, where $q$ is an odd positive integer and $k\geq 1$ is a positive integer. This is equivalent to saying that $x$ admits a binary expansion with finitely many terms. If a number $x>0$ has a non-unique binary expansion, then it has exactly two of them, one with finitely many terms, the other one with infinitely many terms. For example, the number $3/2$ has exactly two binary expansions, namely $3/2=1+1/2=1+\sum_{n=2}^{\infty}2^{-n}$. 
\end{remark}

\begin{proposition}\label{prop_twodiff}
Let $1/2\leq x\leq 1$. There is a unique $\vec{\vartheta}=(\vartheta_{1},\vartheta_{2},\ldots)\in\mathcal{S}$ with $\vartheta_{1}=x$ if and only if the number $1/x$ cannot be written in the form $q/2^{k}$, where $q$ is a  positive odd integer and $k\geq 1$ is a positive integer. 

If $1/x=q/2^{k}$, where $q$ and $k$ are positive integers and $q$ is odd, then there are exactly two different vectors $\vec{\vartheta}\in\mathcal{S}$ such that $\vartheta_{1}=x$. If $q=2^{n_{1}}+2^{n_{2}}+\cdots+2^{n_{m}}$ is the (finite) binary expansion of $q$, then these vectors are 
\begin{align}
\vec{\vartheta} & =\left(x,\frac{x}{2^{k_{2}}},\frac{x}{2^{k_{3}}},\ldots,\frac{x}{2^{k_{m}}},0, 0,\ldots,0,\ldots\right),\label{8-8-2}\\
\vec{\vartheta} & =\left(x,\frac{x}{2^{k_{2}}},\frac{x}{2^{k_{3}}},\ldots,\frac{x}{2^{k_{m-1}}},\frac{x}{2^{k_{m}+1}},\frac{x}{2^{k_{m}+2}},\ldots,\frac{x}{2^{k_{m}+l}},\ldots\right),\label{8-8-3}
\end{align}
(the first one with finitely many non-zero components, the second one with all components non-zero) where $k_{j}=n_{1}-n_{j}$ for each $2\leq j\leq m$.
\end{proposition}
\begin{proof}
Suppose that $1/2\leq x\leq 1$ is such that $1/x$ cannot be written in the form $q/2^{k}$ as indicated. If $x=1$, then it is obvious that there is a unique $\vec{\vartheta}\in\mathcal{S}$ such that $\vartheta_{1}=1$. If $x=1/2$ (and so $1/x=2$), then it is clear that there is a unique $\vec{\vartheta}\in\mathcal{S}$ such that $\vartheta_{1}=1/2$, namely
\[
\vec{\vartheta}=(2^{-1}, 2^{-2},\ldots,2^{-n},\ldots).
\]
If $1/2<x<1$, then we know by Proposition \ref{prop_characS} that there exists $\vec{\vartheta}\in\mathcal{S}$ such that $\vartheta_{1}=x$, and any such $\vec{\vartheta}$ is of the form
\[
\vec{\vartheta}=\left(x,\frac{x}{2^{k_{2}}},\frac{x}{2^{k_{3}}},\ldots\right),
\]
where 
\begin{equation}\label{8-8-1}
\frac{1}{x}=1+\frac{1}{2^{k_{2}}}+\frac{1}{2^{k_{3}}}+\cdots.
\end{equation}
But by Remark \ref{rmk_binexp}, such binary expansion \eqref{8-8-1} is unique, and therefore there is a unique $\vec{\vartheta}\in\mathcal{S}$ such that $\vartheta_{1}=x$.

Now assume that $1/2\leq x\leq 1$ is such that $1/x=q/2^{k}$, where $q$ and $k$ are positive integers and $q$ is odd. Then $1<1/x<2$. If $q=2^{n_{1}}+2^{n_{2}}+\cdots+2^{n_{m}}$ is the finite binary expansion of $q$, then $n_{m}=0$ and we must have $n_{1}=k$. Thus
\[
\frac{1}{x}=1+\frac{1}{2^{k_{2}}}+\frac{1}{2^{k_{3}}}+\cdots+\frac{1}{2^{k_{m}}},
\]
where $k_{j}=n_{1}-n_{j}$ for $2\leq j\leq m$. The only other binary expansion of $1/x$ is
\[
\frac{1}{x}=1+\frac{1}{2^{k_{2}}}+\cdots+\frac{1}{2^{k_{m-1}}}+\sum_{l=1}^{\infty}\frac{1}{2^{k_{m}+l}},
\]   
therefore, by Proposition \ref{prop_characS} there are exactly two vectors $\vec{\vartheta}\in\mathcal{S}$ such that $\vartheta_{1}=x$, which are given by \eqref{8-8-2} and \eqref{8-8-3}.
\end{proof}

\begin{proposition}\label{prop_uniqvalue}
Let $1/2\leq x\leq 1$ be a number for which there are two different vectors $\vec{\vartheta}=(\vartheta_{1}, \vartheta_{2},\ldots)\in\mathcal{S}$ and $\vec{\theta}=(\theta_{1},\theta_{2},\ldots)\in\mathcal{S}$ such that $x=\vartheta_{1}=\theta_{1}$. Then we have $K(\vec{\vartheta})=K(\vec{\theta})$, $R(\vec{\vartheta})=R(\vec{\theta})$, and $H(\vec{\vartheta};s)=H(\vec{\theta};s)$ for any $s>-1$. 
\end{proposition}
\begin{proof}
We know by Proposition \ref{prop_twodiff} that one of the vectors has finitely many non-zero components, say $m$ of them:
\[
\vec{\theta}=(\theta_{1},\ldots,\theta_{m},0, 0,\ldots),
\]
where $m\geq 2$, and the other vector $\vec{\vartheta}=(\vartheta_{1},\vartheta_{2},\ldots)$ satisfies the conditions 
\begin{align}
\vartheta_{j} & =\theta_{j}, \qquad 1\leq j\leq m-1,\label{10-8-6} \\
\vartheta_{j} & =\frac{\theta_{m}}{2^{j-m+1}},\qquad j\geq m. \label{10-8-3}
\end{align}
Note that
\begin{equation}\label{10-8-2}
\sum_{j=m}^{\infty}\vartheta_{j}=\theta_{m}.
\end{equation}
If we take the function
\[
f(x)=\sum_{n=m}^{\infty}x^{n}=\frac{x^{m}}{1-x},\qquad |x|<1,
\]
then
\[
x f'(x)=\sum_{n=m}^{\infty}nx^{n}=\frac{(m(1-x)+x)\,x^{m}}{(1-x)^{2}}.
\]
Evaluating at $x=1/2, 1/4$, we obtain
\begin{equation}\label{10-8-1}
\sum_{n=m}^{\infty}\frac{n}{2^{n}}=\frac{m+1}{2^{m-1}},\qquad\qquad \sum_{n=m}^{\infty}\frac{n}{4^{n}}=\frac{3m+1}{9}\frac{1}{4^{m-1}}.
\end{equation}
We have
\begin{align}
H(\vec{\vartheta};s) & =\sum_{n=1}^{\infty}\vartheta_{n}^{s+1}+2(2^{s}-1)\sum_{n=1}^{\infty}\sum_{j=n+1}^{\infty}\vartheta_{n}^{s}\,\vartheta_{j}\notag\\
& =\sum_{n=1}^{m-1}\theta_{n}^{s+1}+\sum_{n=m}^{\infty}\vartheta_{n}^{s+1}+2(2^{s}-1)\left(\sum_{n=1}^{m-1}\sum_{j=n+1}^{\infty}\theta_{n}^{s}\,\vartheta_{j}+\sum_{n=m}^{\infty}\sum_{j=n+1}^{\infty}\vartheta_{n}^{s}\,\vartheta_{j}\right).\label{10-8-4}
\end{align}
It follows from \eqref{10-8-3} that
\[
\sum_{n=m}^{\infty}\vartheta_{n}^{s+1}=\frac{\theta_{m}^{s+1}}{2^{s+1}-1}.
\]
From \eqref{10-8-6} and \eqref{10-8-2} we deduce that
\[
\sum_{n=1}^{m-1}\sum_{j=n+1}^{\infty}\theta_{n}^{s}\,\vartheta_{j}=\sum_{n=1}^{m-1}\sum_{j=n+1}^{m}\theta_{n}^{s}\,\theta_{j}.
\]
Using \eqref{10-8-3} we obtain
\[
\sum_{n=m}^{\infty}\sum_{j=n+1}^{\infty}\vartheta_{n}^{s}\,\vartheta_{j}=\frac{\theta_{m}^{s+1}}{2^{s+1}-1}.
\]
So applying the previous three identities, and continuing the computation in \eqref{10-8-4}, we obtain
\begin{align*}
H(\vec{\vartheta};s) & =\sum_{n=1}^{m-1}\theta_{n}^{s+1}+\frac{\theta_{m}^{s+1}}{2^{s+1}-1}+2(2^{s}-1)\sum_{n=1}^{m-1}\sum_{j=n+1}^{m}\theta_{n}^{s}\,\theta_{j}+\frac{2(2^{s}-1)\,\theta_{m}^{s+1}}{2^{s+1}-1}\\
& =\sum_{n=1}^{m}\theta_{n}^{s+1}+2(2^{s}-1)\sum_{n=1}^{m-1}\sum_{j=n+1}^{m}\theta_{n}^{s}\,\theta_{j}\\
& =H(\vec{\theta};s).
\end{align*}

We have
\begin{equation}\label{10-8-7}
R(\vec{\vartheta})=-\sum_{n=1}^{m-1}(\log(4)(n-1)+\log\theta_{n})\,\theta_{n}-\sum_{n=m}^{\infty}(\log(4)(n-1)+\log\vartheta_{n})\,\vartheta_{n}.
\end{equation}
Applying \eqref{10-8-3} and the first identity in \eqref{10-8-1}, we obtain
\begin{align}
\sum_{n=m}^{\infty}(\log(4)(n-1)+\log\vartheta_{n})\,\vartheta_{n} & =\sum_{n=m}^{\infty}\log(4^{n-1}\,\vartheta_{n})\,\vartheta_{n}=\sum_{n=m}^{\infty}\log(\theta_{m}\,2^{n+m-3})\frac{\theta_{m}}{2^{n-m+1}}\notag\\
& =\theta_{m}\,2^{m-1}\left(\sum_{n=m}^{\infty}\log(\theta_{m}\,2^{m-3})\frac{1}{2^{n}}+\log(2)\sum_{n=m}^{\infty}\frac{n}{2^{n}}\right)\notag\\
& =\theta_{m}\,2^{m-1}\left(\log(\theta_{m}\,2^{m-3})\frac{1}{2^{m-1}}+\log(2)\frac{m+1}{2^{m-1}}\right)\notag\\
& =\theta_{m}\log(\theta_{m})+\theta_{m}\log(4)(m-1).\label{10-8-8}
\end{align}
Therefore, from \eqref{10-8-7} and \eqref{10-8-8} it follows that $R(\vec{\vartheta})=R(\vec{\theta})$.

We have
\begin{align}
K(\vec{\vartheta})  ={} &\log(4)+\sum_{n=1}^{m-1}\theta_{n}^{2}\log(\theta_{n}/4)+\sum_{n=m}^{\infty}\vartheta_{n}^{2}\log(\vartheta_{n}/4)+2\sum_{n=1}^{m-1}\sum_{j=n+1}^{\infty}\theta_{n}\log(\theta_{n})\,\vartheta_{j}\notag\\& +2\sum_{n=m}^{\infty}\sum_{j=n+1}^{\infty}\vartheta_{n}\log(\vartheta_{n})\,\vartheta_{j}.\label{10-8-12}
\end{align}
Observe that
\begin{equation}\label{10-8-11}
\sum_{n=1}^{m-1}\sum_{j=n+1}^{\infty}\theta_{n}\log(\theta_{n})\,\vartheta_{j}=\sum_{n=1}^{m-1}\sum_{j=n+1}^{m}\theta_{n}\log(\theta_{n})\,\theta_{j}.
\end{equation}
Since $\sum_{j=n+1}^{\infty}\vartheta_{j}=\vartheta_{n}$ for each $n\geq m$, we can write
\begin{equation}\label{10-8-9}
\sum_{n=m}^{\infty}\vartheta_{n}^{2}\log(\vartheta_{n}/4)+2\sum_{n=m}^{\infty}\sum_{j=n+1}^{\infty}\vartheta_{n}\log(\vartheta_{n})\,\vartheta_{j}=\sum_{n=m}^{\infty}(3\,\vartheta_{n}^{2}\log(\vartheta_{n})-\log(4)\,\vartheta_{n}^{2}).
\end{equation}
From \eqref{10-8-3} and the second identity in \eqref{10-8-1} we deduce that
\[
\sum_{n=m}^{\infty}\vartheta_{n}^{2}=\frac{\theta_{m}^{2}}{3}
\]
and
\begin{align*}
\sum_{n=m}^{\infty}\vartheta_{n}^{2}\log(\vartheta_{n}) & =\theta_{m}^{2}\,4^{m-1}\sum_{n=m}^{\infty}\frac{1}{4^{n}}(\log(\theta_{m}\,2^{m-1})-n\log 2)\\
& =\theta_{m}^{2}\,4^{m-1}\left(\frac{1}{3}\frac{\log(\theta_{m}\,2^{m-1})}{4^{m-1}}-\frac{\log(2)}{4^{m-1}}\frac{3m+1}{9}\right)\\
& =\frac{\theta_{m}^{2}}{3}\,(\log(\theta_{m})-\frac{4}{3}\log 2).
\end{align*}
Therefore, continuing the computation in \eqref{10-8-9} we obtain
\begin{align}
\sum_{n=m}^{\infty}\vartheta_{n}^{2}\log(\vartheta_{n}/4)+2\sum_{n=m}^{\infty}\sum_{j=n+1}^{\infty}\vartheta_{n}\log(\vartheta_{n})\,\vartheta_{j} & =\theta_{m}^{2}\log(\theta_{m})-\frac{4\,\theta_{m}^{2}\log 2}{3}-\frac{\theta_{m}^{2}\log 4}{3}\notag\\
& =\theta_{m}^{2}\log(\theta_{m}/4).\label{10-8-10}
\end{align}
From \eqref{10-8-12}, \eqref{10-8-11}, and \eqref{10-8-10} we finally conclude that $K(\vec{\vartheta})=K(\vec{\theta})$. 
\end{proof}

\section{Convexity of the function $H(\vec{\vartheta};s)$}\label{Convexity-H}

In this section we prove that $H(\vec{\vartheta};s)$ is a convex function of $s$ for any $\vec{\vartheta}\in\mathcal{S}$. Let 
\[
\vec{\vartheta}=(\theta_{1},\theta_{2},\theta_{3}, \ldots,\theta_{n},\ldots)\in\mathcal{S}.
\]
Part of the difficulty lies in interpreting correctly the expression \eqref{19-7-1} that defines the function $H$. Let us introduce the notation 
\begin{equation}\label{def:bk}
b_k:=\sum_{j=k+1}^{\infty}\theta_j,\qquad k\geq 1.
\end{equation}
To some extent, it seems that we should write (and think of) \eqref{19-7-1}  as 
\begin{align}\label{Hdef1}
H(\vec{\vartheta};s)=\sum_{k=1}^{\infty}(2\theta_k)^s(2b_k)+\sum_{k=1}^{\infty}\theta_k^s\l(\theta_k-2b_k\r)
\end{align}
because the sum $\sum \l((2\theta_k)^s(2b_k)+\theta_k^s\l(\theta_k-2b_k\r)\r)$ telescopes when taken over geometric sections (of ratio $1/2$) of the vector $\vec{\vartheta}=(\theta_1,\ldots,\theta_n,\ldots)$.

To understand this better, we first establish a lemma.
 \begin{lemma}\label{lemma2} 
Suppose that for some integers $m, r\geq 1$ we have
	\begin{align*}
		\l(\theta_m,\theta_{m+1},\ldots, \theta_{m+r-1}\r)
		={}&\l(\theta_{m}, 2^{-1}\theta_{m}, 2^{-2}\theta_{m},\ldots, 2^{-r+1}\theta_{m}\r).
	\end{align*}
	That is, 
	\begin{align}\label{rel0}
		\theta_k={} & 2^{m-k}\theta_{m},\qquad m\leq k\leq m+r-1.
	\end{align}
	Then, the following relations hold:	
	\begin{align}\label{rel1}
		\theta_{k-1}={} &2\theta_k, \qquad m+1\leq k\leq m+r-1,
	\end{align}
	\begin{align}\label{rel2}
		b_k={} &\theta_k+b_{m+r-1}-\theta_{m+r-1},\qquad m\leq k\leq m+r-1,
	\end{align}
	\begin{align}\label{rel4}
		\theta_k-2b_k=-(b_k+b_{m+r-1}-\theta_{m+r-1}), \qquad m\leq k\leq m+r-1,
	\end{align}
	\begin{align}\label{rel3}
		2b_k=b_{k-1}+b_{m+r-1}-\theta_{m+r-1}, \qquad m+1\leq k\leq m+r-1 .
	\end{align}
	\end{lemma}
	\begin{proof}	
		Relation \eqref{rel1} is an obvious consequence of \eqref{rel0}. To obtain \eqref{rel2}, we write 
		\[
			b_k=\sum_{j=k+1}^{m+r-1}\theta_j+b_{m+r-1}=\frac{\theta_{k+1}-2^{-1}\theta_{m+r-1}}{2^{-1}}+b_{m+r-1}=\theta_{k}-\theta_{m+r-1}+b_{m+r-1}.		
			\]
		Relation \eqref{rel4} follows directly from \eqref{rel2}, and \eqref{rel3} is obtained as follows:
		\[
		2b_k=\sum_{j=k+1}^{m+r-1}2\theta_j+2b_{m+r-1}=\sum_{j=k}^{m+r-2}\theta_j+2b_{m+r-1}=b_{k-1}+b_{m+r-1}-\theta_{m+r-1}.
		\] 
	\end{proof}
	
	Suppose now that the section $\l(\theta_m,\theta_{m+1},\ldots, \theta_{m+r-1}\r)$ of the vector $\vec{\vartheta}$ is as in Lemma \ref{lemma2} above. Then using the relations \eqref{rel1}--\eqref{rel3} we find
	\begin{align}
&	\sum_{k=m}^{m+r-1}(2\theta_k)^s(2b_k)+\sum_{k=m}^{m+r-1}\theta_k^s\l(\theta_k-2b_k\r)\nonumber\\
{} & =(2\theta_m)^s(2b_m)+	\sum_{k=m+1 }^{m+r-1}(2\theta_k)^s(2b_k)+\sum_{k=m}^{m+r-1}\theta_k^s\l(\theta_k-2b_k\r)\nonumber\\
{} & =(2\theta_m)^s(2b_m)+	\sum_{k=m+1 }^{m+r-1}(\theta_{k-1})^s(b_{k-1}+b_{m+r-1}-\theta_{m+r-1})-\sum_{k=m}^{m+r-1}\theta_k^s(b_k+b_{m+r-1}-\theta_{m+r-1})\nonumber\\
{} & =(2\theta_m)^s(2b_m)+\theta_{m+r-1}^s(\theta_{m+r-1}-2b_{m+r-1}).	\label{telescopingproperty}
\end{align}
This is the telescoping property we were previously referring to. If $(\theta_{m},\theta_{m+1},\ldots)$ is an infinite section of the vector $\vec{\vartheta}$ such that $\theta_{k+1}=\theta_{k}/2$ for all $k\geq m$, then taking $r\rightarrow\infty$ in \eqref{telescopingproperty} we obtain
\begin{equation}\label{telescopingproperty2}
\sum_{k=m}^{\infty}(2\theta_k)^s(2b_k)+\sum_{k=m}^{\infty}\theta_k^s\l(\theta_k-2b_k\r)=(2\theta_{m})^{s+1}.
\end{equation}
The condition guaranteeing that $\theta_{m+r-1}-2b_{m+r-1}\geq 0$ is given in the following lemma.

\begin{lemma}\label{lemma1} Let $\vec{\vartheta}=(\theta_{1},\theta_{2},\theta_{3},\ldots)\in\mathcal{S}$. If every component of $\vec{\vartheta}$ is different from zero, then for any $k\geq 1$ we have $\theta_{k}\geq 2 b_{k}$ if and only if $\theta_{k}>2\theta_{k+1}$. If $\vec{\vartheta}$ has finitely many non-zero components and $p=\max\{k\geq 1: \theta_{k}>0\}$, then for $1\leq k\leq p$ we have $\theta_{k}\geq 2 b_{k}$ if and only if $k\geq p-1$ or $\theta_{k}>2\theta_{k+1}$.
\end{lemma}
\begin{proof}
Suppose first that every component of $\vec{\vartheta}$ is different from zero. Assume that $\theta_{k}\leq 2\theta_{k+1}$. Then $$2 b_{k}=2\sum_{j=k+1}^{\infty}\theta_{j}>2\theta_{k+1}\geq \theta_{k}$$ and therefore $2b_{k}>\theta_{k}$. Suppose now that $\theta_{k}>2\theta_{k+1}$. By Proposition \ref{prop_19-7-4}, we have $\theta_{k+1}=2^{-r}\theta_{k}$ for some integer $r\geq 1$, which implies that $\theta_{k+1}\leq \theta_{k}/4$. Hence, again applying Proposition \ref{prop_19-7-4} we obtain
\[
2b_{k}=2\sum_{j=k+1}^{\infty}\theta_{j}\leq 2\sum_{j=0}^{\infty}\frac{\theta_{k+1}}{2^{j}}=4\theta_{k+1}\leq \theta_{k}
\]
and this concludes the proof of the first statement.

Suppose now that $\vec{\vartheta}$ has finitely many non-zero components and $p=\max\{k\geq 1: \theta_{k}>0\}$. If $k=p$, then trivially $\theta_{p}\geq 2 b_{p}=0$. If $p\geq 2$ and $k=p-1$, then $\theta_{p-1}\geq 2\theta_{p}=2 b_{p-1}$. If $\theta_{k}>2\theta_{k+1}$, then as before we have $\theta_{k}\geq 4\theta_{k+1}$ and 
\[
2 b_{k}=2\sum_{j=k+1}^{p}\theta_{j}\leq 2\sum_{j=0}^{p-k-1}\frac{\theta_{k+1}}{2^{j}}\leq 4\theta_{k+1}\leq \theta_{k}.
\]
Suppose now that $\theta_{k}\geq 2 b_{k}$ and $k\leq p-2$. If $\theta_{k}\leq 2\theta_{k+1}$, then $2 b_{k}=2\sum_{j=k+1}^{p}\theta_{j}>2\theta_{k+1}\geq \theta_{k}$, which is a contradiction.	
\end{proof}

In what follows, by a \emph{string of integers} we mean a finite or infinite set of consecutive integers. Given a vector $\vec{\vartheta}=(\theta_{1},\theta_{2},\ldots)\in\mathcal{S}$, we know by Proposition \ref{prop_characS} that $1/2\leq \theta_{1}\leq 1$, and for every $k$ such that $\theta_{k}>0$ we have $\theta_{k}=\theta_{1} 2^{-n_{k}}$, where $0=n_{1}<n_{2}<n_{3}<\cdots$ is an increasing sequence of non-negative integers such that
\begin{equation}\label{exptheta1}
\frac{1}{\theta_{1}}=1+\frac{1}{2^{n_{2}}}+\frac{1}{2^{n_{3}}}+\cdots.
\end{equation}
If the expansion \eqref{exptheta1} is infinite, then
\begin{equation}\label{expinfinite}
\vec{\vartheta}=\left(\theta_{1},\frac{\theta_{1}}{2^{n_{2}}},\frac{\theta_{1}}{2^{n_{3}}},\ldots,\frac{\theta_{1}}{2^{n_{k}}},\ldots\right),
\end{equation}
and if the expansion \eqref{exptheta1} is finite and has $p$ terms, then 
\begin{equation}\label{expfinite}
\vec{\vartheta}=\left(\theta_{1},\frac{\theta_{1}}{2^{n_{2}}},\frac{\theta_{1}}{2^{n_{3}}},\ldots,\frac{\theta_{1}}{2^{n_{p}}},0,0,\ldots\right).
\end{equation}

Suppose that \eqref{expfinite} holds. Then there is a unique partition $\caliP(\vec{\vartheta})=\{J_1,J_2,\ldots, J_q\}$ of  $\{1,2,\ldots,p\}$ into non-empty sets $J_{\ell}$ satisfying the following conditions:
\begin{enumerate}
	\item[$c1)$] For each $1\leq \ell\leq q$, both $J_\ell\subset\{1,\ldots,p\}$ and the set $\{n_k:k\in J_\ell\}\subset\mathbb{N}\cup\{0\}$ are finite strings of non-negative integers. That is to say,
	\[
	J_\ell=\{m_\ell,m_\ell+1,\ldots,m_\ell+r_\ell-1\}
	\]
	and   
	\[
	n_{m_\ell+r_\ell-1}=n_{m_\ell}+r_\ell-1.
	\]  
	
	\item[$c2)$] $m_\ell+r_\ell=m_{\ell+1}$ for all $1\leq \ell< q$.
	\item[$c3)$] For each $1\leq \ell< q$, we have 
	\begin{align}\label{property3}
		2^{-n_{m_\ell+r_\ell-1}} \geq 2^{-n_{m_{\ell+1}}+2}.
	\end{align}
\end{enumerate}

As an example, if $\vec{\vartheta}=\eta(N)$ for the number 
	\[
N=2^{13}+2^{12}+2^{10}+2^8+2^7+2^6+2^3+2+1,
	\]
	then $\theta_{1}=2^{13}/N$, 
	\[
	\vec{\vartheta}=\theta_{1}\left(1,2^{-1},2^{-3},2^{-5},2^{-6},2^{-7},2^{-10},2^{-12},2^{-13},0,0,\ldots\right),
	\]
and we have $\caliP(\vec{\vartheta})=\{J_1,J_2,J_3,J_4,J_5\}$ with $J_1=\{1,2\}$, $J_2=\{3\}$, $J_3=\{4,5,6\}$, $J_4=\{7\}$, $J_5=\{8,9\}$.

Then, using \eqref{telescopingproperty} we can write 
\begin{align}
	H(\vec{\vartheta};s)={}&\sum_{\ell=1}^q\l(\sum_{k\in J_\ell}(2\theta_k)^s(2b_k)+\sum_{k\in J_\ell}\theta_k^s\l(\theta_k-2b_k\r)\r)\notag\\
	={} &\sum_{\ell=1}^q\l((2\theta_{m_\ell})^s(2b_{m_\ell})+\theta_{m_\ell+r_\ell-1}^s(\theta_{m_\ell+r_\ell-1}-2b_{m_\ell+r_\ell-1})\r).\label{16-7-1}
\end{align}
Note  that by \eqref{property3}, we have that either $m_\ell+r_\ell-1=p$, or $\theta_{m_\ell+r_\ell-1}>2 \theta_{m_\ell+r_\ell}$, so that by Lemma \ref{lemma1}, it is the case that   
\begin{align}\label{ineq1}
	\theta_{m_\ell+r_\ell-1}-2b_{m_\ell+r_\ell-1}\geq 0,\qquad 1\leq \ell\leq q.	
\end{align}
We conclude that if $\vec{\vartheta}$ is of the form \eqref{expfinite}, then $H(\vec{\vartheta};s)$ is a finite sum of convex functions of $s$, and so it is a convex function. Moreover, if $\vec{\vartheta}\neq (1,0,0,\ldots,0,\ldots)$, then $p\geq 2$ and $\theta_{1}>1/2$, hence $b_{1}>0$ and the function $(2\theta_{1})^{s}(2 b_{1})$ is strictly convex. Since $m_{1}=1$, it follows from \eqref{16-7-1} that in this case  
$H(\vec{\vartheta};s)$ is a strictly convex function of $s$ on the whole real line $(-\infty,\infty)$. If $\vec{\vartheta}=(1,0,0,\ldots)$, then $H(\vec{\vartheta};s)\equiv 1$.

Now suppose that \eqref{expinfinite} holds, that is, every component of $\vec{\vartheta}$ is different from zero. Then there is a unique partition $\mathcal{P}(\vec{\vartheta})=\{J_{1},J_{2},\ldots\}$ of $\mathbb{N}$, where the number $q\in\mathbb{N}\cup\{\infty\}$ of sets $J_{\ell}$ in the partition is finite or infinite, satisfying the following properties:

\begin{enumerate}
\item For each $1\leq \ell<q$, $J_{\ell}$ and $\{n_{k}: k\in J_{\ell}\}$ are finite strings. Thus
\[
	J_\ell=\{m_\ell,m_\ell+1,\ldots,m_\ell+r_\ell-1\}
	\]
	and   
	\[
	n_{m_\ell+r_\ell-1}=n_{m_\ell}+r_\ell-1.
	\]
	
	\item If $q$ is finite, then $J_{q}=\{m_{q}, m_{q}+1,\ldots\}$ and $\{n_{k}: k\in J_{q}\}$ are infinite strings.
	
	\item The above conditions $c2)$ and $c3)$ hold. 
\end{enumerate}

Then, if $q$ is finite, applying \eqref{telescopingproperty} and \eqref{telescopingproperty2} we obtain in a similar way the identity
\begin{equation}\label{expfinH}
H(\vec{\vartheta};s)=\sum_{\ell=1}^{q-1}\l((2\theta_{m_\ell})^s(2b_{m_\ell})+\theta_{m_\ell+r_\ell-1}^s(\theta_{m_\ell+r_\ell-1}-2b_{m_\ell+r_\ell-1})\r)+(2\theta_{m_{q}})^{s+1},
\end{equation}
which shows that $H(\vec{\vartheta};s)$ is a convex function of $s$. If $\vec{\vartheta}=(2^{-1},2^{-2},2^{-3},\ldots)$, then $q=1$, $J_{1}=\mathbb{N}$, and therefore $H(\vec{\vartheta};s)\equiv 1$. If $\vec{\vartheta}\neq (2^{-1},2^{-2},2^{-3},\ldots)$ and $q\geq 2$ is finite, then $\theta_{1}>1/2$ and $(2\theta_{1})^{s}(2b_{1})$ is strictly convex, therefore \eqref{expfinH} implies in this case that $H(\vec{\vartheta};s)$ is strictly convex.

If $q=\infty$, then every set $J_{\ell}$ is finite and we have
\[
H(\vec{\vartheta};s)=\sum_{\ell=1}^{\infty}\l((2\theta_{m_\ell})^s(2b_{m_\ell})+\theta_{m_\ell+r_\ell-1}^s(\theta_{m_\ell+r_\ell-1}-2b_{m_\ell+r_\ell-1})\r).
\]
Since $H(\vec{\vartheta};s)$ is a convergent series of convex functions and $(2\theta_{1})^{s}(2b_{1})$ is strictly convex, it follows that $H(\vec{\vartheta};s)$ is strictly convex. 

Thus we have proven the following result:
\begin{theorem}\label{theo:convexH}
	The function $s\mapsto H(\vec{\vartheta};s)$ is strictly convex on the interval $(-1,\infty)$ for any $\vec{\vartheta}\in\mathcal{S}$ different from the vectors $(1,0,0,\ldots)$ and $(2^{-1},2^{-2},2^{-3},\ldots)$. If $\vec{\vartheta}$ is equal to one of these two vectors, then $H(\vec{\vartheta};s)\equiv 1$.
\end{theorem}

Since $H(\vec{\vartheta};0)=H(\vec{\vartheta};1)=1$, we obtain as a corollary that
\begin{align}\label{convexitybounds}\begin{split}
	H(\vec{\vartheta};s) & \leq 1,\qquad 0\leq s\leq 1,\quad \vec{\vartheta}\in\mathcal{S},\\
	H(\vec{\vartheta};s) & \geq  1,\qquad s\in(-1,\infty)\setminus[0,1],\quad \vec{\vartheta}\in\mathcal{S} .
	\end{split}
\end{align}

\section{Continuity of the functions $\mathcal{H}$, $\mathcal{K}$, $\mathcal{R}$, and density of the limit points}\label{continuity-functions-H-K-R}

\begin{definition}\label{deffuncHKRhat}
We define three functions as follows. For $1/2\leq x\leq 1$ and $s>-1$, let
\begin{align*}
\mathcal{H}(x,s) & :=H(\vec{\vartheta};s)\\
\mathcal{K}(x) & :=K(\vec{\vartheta})\\
\mathcal{R}(x) &:=R(\vec{\vartheta})
\end{align*}
where $\vec{\vartheta}=(\vartheta_{1},\vartheta_{2},\vartheta_{3},\ldots)$ is any vector in $\mathcal{S}$ such that $\vartheta_{1}=x$. In virtue of Proposition \ref{prop_uniqvalue}, these functions are well-defined.  
\end{definition}

For $1/2\leq x<1$, we write 
\begin{equation}\label{firstform}
\vec{\vartheta}_\infty(x):=\left(x,\frac{x}{2^{k_{2}}},\frac{x}{2^{k_{3}}},\ldots\right),
\end{equation}
where 
\[
\frac{1}{x}=\frac{1}{2^{k_{1}}}+\frac{1}{2^{k_{2}}}+\cdots+\frac{1}{2^{k_{j}}}+\cdots,\qquad 0=k_1<k_2<\cdots,
\]
is the \emph{infinite} binary expansion of $1/x$. If $1/2<x\leq 1$ and $1/x$ has a \emph{finite} binary expansion   	
\[
\frac{1}{x}=1+\frac{1}{2^{k_{2}}}+\cdots+\frac{1}{2^{k_{m}}},
\]
then we set 
\begin{align}\label{secondform}
\vec{\vartheta}(x):=\left(x,\frac{x}{2^{k_{2}}},\frac{x}{2^{k_{3}}},\ldots, \frac{x}{2^{k_{m}}},0,0,\ldots\right).
\end{align}
Note that for an $x\in (1/2,1)$ having two associated vectors \eqref{firstform} and \eqref{secondform}, we have  
\begin{align}\label{relation-two-vectors}
\vec{\vartheta}_\infty(x)=\left(x,\frac{x}{2^{k_{2}}},\frac{x}{2^{k_{3}}},\ldots, \frac{x}{2^{k_{m-1}}},\frac{x}{2^{k_{m}+1}},\frac{x}{2^{k_{m}+2}},\ldots,\frac{x}{2^{k_{m}+r}},\ldots \right),
\end{align}
and it is the assertion of Proposition \ref{prop_uniqvalue}	that 
\begin{align}\label{same-values-H}
H(\vec{\vartheta}_\infty(x);s)=H(\vec{\vartheta}(x);s),
\end{align}
with this equality holding true as well with the functions $K$ and $R$ in place of $H(\cdot;s)$.

The following two lemmas are elementary. Their proofs are left to the reader.

\begin{lemma}\label{lem_elem1}
	Let $1/2< x\leq 1$, and assume that $1/x$ has a finite binary expansion
	\begin{equation}\label{15-8-1}
		\frac{1}{x}=\frac{1}{2^{k_{1}}}+\frac{1}{2^{k_{2}}}+\cdots+\frac{1}{2^{k_{m}}},\qquad 0=k_{1}<k_{2}<\cdots<k_{m}.
	\end{equation}
\begin{enumerate}	
\item[(i)] If $1/2<y<x$ is so close to $x$ that $1/x<1/y<1/x+2^{-k_{m}-r}$, where $r\geq 0$ is an integer, then any binary expansion of $1/y$ is of the form
	\begin{equation}\label{14-8-1}
		\frac{1}{y}=1+\frac{1}{2^{k_{2}}}+\cdots+\frac{1}{2^{k_{m}}}+\frac{1}{2^{q_{m+1}}}+\frac{1}{2^{q_{m+2}}}+\cdots
	\end{equation}
	where $q_{m+1}>k_{m}+r$. As a consequence, we have
	\[
	\lim_{y\to x-}\|\vec{\vartheta}_\infty(y)-\vec{\vartheta}(x)\|_{\ell^\infty}=	\lim_{y\to x-}\|\vec{\vartheta}(y)-\vec{\vartheta}(x)\|_{\ell^\infty}=0,
	\]
where the second limit is taken along points $y$ such that $1/y$ has a finite binary expansion.
\item[(ii)] 	If $x<y<1$ is so close to $x$ that $1/x-2^{-k_{m}-r}<1/y<1/x$, where $r\geq 1$ is a positive integer, then any binary expansion of $1/y$ is of the form
\[
		\frac{1}{y}=1+\frac{1}{2^{k_{2}}}+\cdots+\frac{1}{2^{k_{m-1}}}+\frac{1}{2^{k_{m}+1}}+\frac{1}{2^{k_{m}+2}}+\cdots+\frac{1}{2^{k_{m}+r}}+\frac{1}{2^{q_{m+r}}}+\frac{1}{2^{q_{m+r+1}}}+\cdots
\]
	where $q_{m+r}>k_{m}+r$.  As a consequence, we have
	\[
	\lim_{y\to x+}\|\vec{\vartheta}_\infty(y)-\vec{\vartheta}_\infty(x)\|_{\ell^\infty}=	\lim_{y\to x+}\|\vec{\vartheta}(y)-\vec{\vartheta}_\infty(x)\|_{\ell^\infty}=0,
	\]
\end{enumerate}
where the second limit is taken along points $y$ such that $1/y$ has a finite binary expansion.	
		\end{lemma}
\begin{lemma}\label{lem_elem2}
		Let $1/2\leq  x<1$ and assume  that $1/x$ does not admit a finite binary expansion of the form \eqref{15-8-1} (which is the case of $x=1/2$). Suppose that 
	\begin{equation}\label{15-8-2}
		\frac{1}{x}=\frac{1}{2^{k_{1}}}+\frac{1}{2^{k_{2}}}+\cdots+\frac{1}{2^{k_{j}}}+\cdots,\qquad 0=k_{1}<k_{2}<\cdots.
	\end{equation}
	Let $t_{m}:=\sum_{l=1}^{m}2^{-k_{l}}$, $m\geq 1$. If $y$ is so close to $x$ that either 
	\[
	 t_{m}<1/y<1/x\qquad or \qquad 1/x<1/y<t_{m}+2^{-k_{m}},
	\]
	then any expansion of $1/y$ is of the form \eqref{14-8-1} with $q_{m+1}>k_{m}$. As a consequence, we have
	\[
	\lim_{y\to x}\|\vec{\vartheta}_\infty(y)-\vec{\vartheta}_\infty(x)\|_{\ell^\infty}=	\lim_{y\to x}\|\vec{\vartheta}(y)-\vec{\vartheta}_\infty(x)\|_{\ell^\infty}=0,
	\]
	where the second limit is taken along points $y$ such that $1/y$ has a finite binary expansion.
\end{lemma}

\begin{theorem}\label{theo:contHKRcal}
The three functions $\mathcal{H}(x,s)$, $s>-1$, $\mathcal{K}(x)$, $\mathcal{R}(x)$ introduced in Definition \ref{deffuncHKRhat} are continuous on the interval $1/2\leq x\leq 1$.
\end{theorem}
\begin{proof}
Suppose that $1/2<x\leq 1$ and $1/x$ has a finite binary expansion of the form \eqref{15-8-1}. By Lemma \ref{lem_elem1}, we have 
\[
\lim_{y\to x-}\|\vec{\vartheta}_\infty(y)-\vec{\vartheta}(x)\|_{\ell^\infty}=	\lim_{y\to x+}\|\vec{\vartheta}_\infty(y)-\vec{\vartheta}_\infty(x)\|_{\ell^\infty}=0.
\]
By the continuity of the function $ H(\vec{\vartheta};s)$ over the set $\mathcal{S}$ (see Proposition \ref{prop_contHKR}) and the equality \eqref{same-values-H}, we conclude that 
	\begin{equation}\label{limconvHcal}
	\lim_{y\to x}|\mathcal{H}(y,s)-\mathcal{H}(x,s)|=0.
	\end{equation}
If $1/2\leq x<1$ and $1/x$ does not admit a finite binary expansion of the form \eqref{15-8-1}, then by Lemma \ref{lem_elem2} we have $\lim_{y\to x}\|\vec{\vartheta}_\infty(y)-\vec{\vartheta}_{\infty}(x)\|_{\ell^\infty}=0$, which implies \eqref{limconvHcal}.

The very same argument can be used to extend the proof to the other two functions $\mathcal{K}(x)$ and $\mathcal{R}(x)$.	
\end{proof}

\noindent\textbf{Proof of Theorem \ref{theo:density-interval}:} We know by Theorem \ref{theo:limpointHKR} that the set of limit points of the sequence $(H(\eta(N);s))_{N}$ is the set $\{H(\vec{\vartheta};s): \vec{\vartheta}\in\mathcal{S}\}$. But Propositions \ref{prop_characS} and \ref{prop_twodiff} immediately imply that $\{H(\vec{\vartheta};s): \vec{\vartheta}\in\mathcal{S}\}=\{\mathcal{H}(x,s): 1/2\leq x\leq 1\}$. Then from Theorem \ref{theo:contHKRcal} we conclude that the set $\{H(\vec{\vartheta};s): \vec{\vartheta}\in\mathcal{S}\}$ is a closed interval. The same argument can be applied to the sequences $(K(\eta(N)))_{N}$ and $(R(\eta(N)))_{N}$. 
\qed

\section{Extremal points of the functions $\mathcal{H}(x,s)$}

For $1/2\leq x\leq 1$ and $s>-1$, we have previously defined 
	\begin{align*}
		\mathcal{H}(x,s) & :=H(\vec{\vartheta};s)
	\end{align*}
	where $\vec{\vartheta}=(\vartheta_{1},\vartheta_{2},\ldots)$ is any vector in $\mathcal{S}$ such that $\vartheta_{1}=x$. By Proposition \ref{prop_uniqvalue}, these functions are well-defined. Because $\mathcal{H}(1/2,s)=\mathcal{H}(1,s)=1$ and the bounds \eqref{convexitybounds}, for every $0<s<1$ the absolute minimum of $\mathcal{H}(x,s)$ is attained at some point of $(1/2,1)$, and the same is true of the absolute maximum of   $\mathcal{H}(x,s)$ for $-1<s<0$, or $s>1$.

Recall that for $1/2< x<1$, we write 
\[
\vec{\vartheta}_\infty(x)=\left(x,\frac{x}{2^{k_{2}}},\frac{x}{2^{k_{3}}},\ldots\right)
\]
if 
\begin{align*}
	\frac{1}{x}=\frac{1}{2^{k_{1}}}+\frac{1}{2^{k_{2}}}+\cdots+\frac{1}{2^{k_{j}}}+\cdots,\qquad 0=k_1<k_2<\cdots
\end{align*}
is the \emph{infinite} binary expansion of $1/x$. If this is the only binary expansion that $1/x$ has, we write $x\in I_1$. Otherwise, $1/x$ has a second, \emph{finite} binary expansion   	
\[
\frac{1}{x}=1+\frac{1}{2^{k_{2}}}+\cdots+\frac{1}{2^{k_{m}}},
\]
in which case we write $x\in I_2$ and  
\begin{align*}
	\vec{\vartheta}(x)=\left(x,\frac{x}{2^{k_{2}}},\frac{x}{2^{k_{3}}},\ldots, \frac{x}{2^{k_{m}}},0,0,\ldots\right).
\end{align*}
Thus, if the point $x\in I_2$ has two associated vectors \eqref{firstform} and \eqref{secondform}, then  
\begin{align*}
	\vec{\vartheta}_\infty(x):=\left(x,\frac{x}{2^{k_{2}}},\frac{x}{2^{k_{3}}},\ldots, \frac{x}{2^{k_{m-1}}},\frac{x}{2^{k_{m}+1}},\frac{x}{2^{k_{m}+2}},\ldots,\frac{x}{2^{k_{m}+r}},\ldots \right),
\end{align*}
and in virtue of Proposition \ref{prop_uniqvalue} we have
\begin{align*}
	H(\vec{\vartheta}_\infty(x),s)=H(\vec{\vartheta}(x),s).
\end{align*}

\begin{theorem}\label{thm-deriv} If $0<s<1$, then every local minimum of $\mathcal{H}(x,s)$ is attained at some point  $x\in I_1$ satisfying the equation 
	\begin{equation}\label{derivative-ecua}
	\mathcal{H}(x,s)=\frac{2(2^{s}-1)}{s+1}\sum_{n=1}^{\infty}\left(\frac{x}{2^{k_n}}\right)^s.
	\end{equation}
If $s>1$, then every local maximum of $\mathcal{H}(x,s)$ is attained at some point $x\in I_1$ satisfying \eqref{derivative-ecua}.
\end{theorem}

Since \eqref{derivative-ecua} is then satisfied at the points where the function $\mathcal{H}(x,s)$ attains its absolute minimun or maximum (depending on the value of $s$), we obtain as a corollary that 
\[
 \frac{2(2^{s}-1)}{s+1}\leq \mathcal{H}(x,s)\leq 1,\qquad 0<s<1, \quad x\in[1/2,1],
\]
\begin{equation}\label{second-bound-H}
1\leq \mathcal{H}(x,s)\leq \frac{2(2^{s}-1)}{(s+1)},\qquad s>1,\quad x\in[1/2,1].
\end{equation}

\noindent\textbf{Proof of Theorem \ref{thm-deriv}:} 
Let us assume that $x\in I_1$ is a point where the function $\mathcal{H}(\cdot,s)$ has a local maximum (if $s>1$) or a local minimum (if $0<s<1$). Let
\[
\vec{\vartheta}_\infty(x)=\left(x,\frac{x}{2^{k_{2}}},\frac{x}{2^{k_{3}}},\ldots, \frac{x}{2^{k_{m-1}}},\frac{x}{2^{k_{m}}},\frac{x}{2^{k_{m+1}}},\ldots\right),
\]
and define  $x^+_m$  by 
\[
\frac{1}{x^+_m}=1+\frac{1}{2^{k_{2}}}+\cdots+\frac{1}{2^{k_{m-1}}}+\frac{1}{2^{k_{m+1}}}+\frac{1}{2^{k_{m+2}}}+\cdots
\]
so that 
\begin{align}\label{releq1}
x^+_m-x=\frac{x x^+_m}{2^{k_m}}>0
\end{align}
and 
\[
\vec{\vartheta}_\infty(x^+_m)=\left(x^+_m,\frac{x^+_m}{2^{k_{2}}},\frac{x^+_m}{2^{k_{3}}},\ldots, \frac{x^+_m}{2^{k_{m-1}}},\frac{x^+_m}{2^{k_{m+1}}},\ldots\right).
\]
By comparing the series  $H(\vec{\vartheta}_\infty(x),s)$ and $H(\vec{\vartheta}_\infty(x^+_m),s)$ that result from   direct evaluations in \eqref{19-7-1}, one obtains the identity
\begin{align*}
(x^+_m)^{-s-1}\mathcal{H}(x^+_m,s) ={}&x^{-s-1}\mathcal{H}(x,s)-2^{-k_m(s+1)}-2(2^{s}-1)2^{-k_m}\sum_{n=1}^{m-1}2^{-k_ns}\\
&-2(2^{s}-1)2^{-k_ms}\sum_{j=m+1}^{\infty}2^{-k_j}.
\end{align*}
From this identity and \eqref{releq1} we find
\begin{align*}
\frac{	\mathcal{H}(x^+_m,s)-\mathcal{H}(x,s)}{x^+_m-x}={} &\frac{\left(\frac{x^+_m}{x}\right)^{s+1}-1}{x^+_m-x}\mathcal{H}(x,s)-\frac{(x^+_m-x)^s}{x^{s+1}}-\frac{( x^+_m)^{s}2(2^{s}-1)}{x }\sum_{n=1}^{m-1}2^{-k_ns}\\
& -2(2^{s}-1)\frac{(x^+_m-x)^s}{x ^{s+1}}\sum_{j=m+1}^{\infty}2^{k_m-k_j}.
\end{align*}
If we take the limit as $m\to\infty$ (recall that $s>0$), we obtain
\begin{align}\label{ymlimit}\begin{split}
x^{-1}\left((s+1)\mathcal{H}(x,s)-x^s2(2^{s}-1)\sum_{n=1}^{\infty}2^{-k_ns}\right)\geq{} & 0,\qquad 0<s<1,\\
x^{-1}\left((s+1)\mathcal{H}(x,s)-x^s2(2^{s}-1)\sum_{n=1}^{\infty}2^{-k_ns}\right)\leq {} & 0,\qquad s>1.
\end{split}
\end{align}
Now, since $x\not\in I_2$, the vector
\[
\vec{\vartheta}_\infty(x)=\left(x,\frac{x}{2^{k_{2}}},\frac{x}{2^{k_{3}}},\ldots\right),
\]
has the property that there are infinitely many values of $m$ such that 
\[
\frac{1}{2^{k_{m-1}+1}}>\frac{1}{2^{k_{m}}}.
\]
For such values of $m$ consider the number $x^-_m$  given by 
\[
\frac{1}{x^-_m}=1+\frac{1}{2^{k_{2}}}+\cdots+\frac{1}{2^{k_{m-1}}}+\frac{1}{2^{k_{m}-1}}+\frac{1}{2^{k_{m+1}}}+\cdots
\]
so that
\begin{align}\label{releq2}
	x^-_m-x=-\frac{x x^-_m}{2^{k_m}}<0.
\end{align}
This time we get the identity
\begin{align*}
	(x^-_m)^{-s-1}\mathcal{H}(x^-_m,s) 
	={} & x^{-s-1}\mathcal{H}(x,s)+(2^{s+1}-1)2^{-k_m(s+1)}\\
	&+2(2^{s}-1)2^{-k_m}\sum_{n=1}^{m-1}2^{-k_ns}+2(2^{s}-1)^22^{-k_ms}\sum_{j=m+1}^{\infty}2^{-k_j}.
\end{align*}
From this identity and \eqref{releq2} we find
\begin{align*}
	\frac{	\mathcal{H}(x^-_m,s)-\mathcal{H}(x,s)}{x^-_m-x}={} &\frac{\left(\frac{x^-_m}{x}\right)^{s+1}-1}{x^-_m-x}\mathcal{H}(x,s)-(2^{s+1}-1)\frac{(x-x^-_m)^s}{x^{s+1}}-(x^-_m)^s\frac{2(2^{s}-1)}{x}\sum_{n=1}^{m-1}2^{-k_ns}\\
	& -2(2^{s}-1)^2\frac{(x-x^-_m)^s}{x^{s+1}}\sum_{j=m+1}^{\infty}2^{k_m-k_j}.
\end{align*}
Taking the limit as $m\to\infty$ we get
\begin{align}\label{vmlimit}\begin{split}
		x^{-1}\left((s+1)\mathcal{H}(x,s)-x^s2(2^{s}-1)\sum_{n=1}^{\infty}2^{-k_ns}\right)\leq{} & 0,\qquad 0<s<1,\\
		x^{-1}\left((s+1)\mathcal{H}(x,s)-x^s2(2^{s}-1)\sum_{n=1}^{\infty}2^{-k_ns}\right)\geq {} & 0,\qquad s>1.
	\end{split}
\end{align}
It follows from \eqref{ymlimit} and \eqref{vmlimit} that \eqref{derivative-ecua} holds.

What remains to be proven is the fact that at no point  $x\in I_2$ can the function $H(\cdot,s)$ attain a local minimum or maximum (depending on the value of $s$). We argue by contradiction. Let us assume that the function $\mathcal{H}(\cdot,s)$ attains a local minimum (resp. a local maximum) at the number $x\in (1/2,1)$ if $0<s<1$ (resp. if $s>1$) and that $1/x$ has a finite binary expansion of the form 
\begin{align*}
	\frac{1}{x}={} &\sum_{j=1}^{p}\frac{1}{2^{k_j}}
\end{align*}	
with 
$0= k_1<k_2<\cdots<k_{p}= M$ and $p\geq 2$. Consider  the number  
\begin{align*}
	\frac{1}{x^-_m}:=	\frac{1}{x}+\frac{1}{2^{k_p+m}}={} &\sum_{j=1}^{p}\frac{1}{2^{k_j}}+\frac{1}{2^{k_p+m}}
\end{align*}
so that $x^-_m<x$ and $\lim_{m\to\infty}x^-_m=x$. Similarly,  
\begin{align*}
	\frac{1}{x^+_{m}}:=	 &\sum_{j=1}^{p-1}\frac{1}{2^{k_j}}+\sum_{j=1}^{m}\frac{1}{2^{k_p+j}}<\frac{1}{x},
\end{align*}
so that $x^+_m>x$ and $\lim_{m\to\infty}x^+_m=x$. Note that   
\[
x^-_m-x=-\frac{xx^-_m}{2^{k_p+m}}.
\]
Associated to the above binary representations we have the vectors 
\begin{align*}
	\vec{\vartheta}(	x)={} &\left(	\frac{x}{2^{k_1}},\frac{x}{2^{k_2}},\ldots,\frac{	x}{2^{k_p}}\right),\\
	\vec{\vartheta}(x^-_m)={} &\left(	\frac{x^-_m}{2^{k_1}},\frac{x^-_m}{2^{k_2}},\ldots,\frac{	x^-_m}{2^{k_p}},\frac{	x^-_m}{2^{k_p+m}}\right),\\
	\vec{\vartheta}(x^+_{m})={} &\left(	\frac{x^+_{m}}{2^{k_1}},\frac{x^+_{m}}{2^{k_2}},\ldots,\frac{x^+_{m}}{2^{k_{p-1}}},	\frac{x^+_{m}}{2^{k_p+1}},\cdots, \frac{x^+_{m}}{2^{k_p+m}}\right).
\end{align*}
By evaluating $ H(\vec{\theta};s)$ at the first two of these vectors and comparing the resulting expressions one finds the identity 
\begin{align*}
	\mathcal{H}(x^-_m,s)	={} &\left (\frac{x^-_m}{x}\right)^{s+1}\mathcal{H}(x,s)
	+(x^-_m)^{s+1}\left(\frac{1}{2^{(k_p+m)(s+1)}}	+\frac{2^{s}-1}{2^{k_p+m-1}}\sum_{j=1}^{p}\frac{1}{2^{k_js}}\right).
\end{align*}
Then
\begin{align*}
	\frac{\mathcal{H}(x^-_m,s)-\mathcal{H}(x,s)}{x^-_m-x}	 &{} =\frac{\left (\frac{x^-_m}{x}\right)^{s+1}-1}{x^-_m-x}\mathcal{H}(x,s)-\frac{(x^-_m)^{s+1}\left(\frac{1}{2^{(k_p+m)(s+1)}}	+\frac{2^{s}-1}{2^{k_p+m-1}}\sum_{j=1}^{p}\frac{1}{2^{k_js}}\right)}{\frac{xx^-_m}{2^{k_p+m}}}.
\end{align*}
By taking the limit as $m\to\infty$, we obtain
\begin{align}
(s+1)\mathcal{H}(x,s)-x^{s}2(2^{s}-1)\sum_{j=1}^{p}\frac{1}{2^{k_js}}\leq {} &0,\qquad 0<s<1,\label{ecua19}\\
(s+1)\mathcal{H}(x,s)-x^{s}2(2^{s}-1)\sum_{j=1}^{p}\frac{1}{2^{k_js}}	\geq {} &0,  \qquad s>1.\label{ecua20}
\end{align}
Similarly, by evaluating $ H(\vec{\theta};s)$ at $\vec{\theta}= 	\vec{\vartheta}(x)$ and at $\vec{\theta}= 	\vec{\vartheta}(x^+_m)$ and comparing the resulting expressions we find that
\begin{align*}
	x^{-s-1}	\mathcal{H}(x,s)	={} &(x^+_m)^{-s-1}	\mathcal{H}(x^+_m,s)\\
	&+a_{s,m}+2(2^{s}-1)a_m\sum_{j=1}^{p-1}\frac{1}{2^{k_js}}-2(2^{s}-1)\sum_{i=1}^{m-1}\frac{1}{2^{(k_p+i)s}}\sum_{l=i+1}^{m}\frac{1}{2^{k_p+l}},
\end{align*}
where 
\[
a_m:=\frac{1}{x}-\frac{1}{x^+_m}=\left(\frac{1}{2^{k_p}}-\sum_{i=1}^m\frac{1}{2^{k_p+i}}\right)=\frac{1}{2^{k_p+m}},
\]
and (as $m\to\infty$)
\begin{align*}
a_{s,m}:=\left(\frac{1}{2^{k_p(s+1)}}-\sum_{i=1}^m\frac{1}{2^{(k_p+i)(s+1)}}\right)={} &\frac{2(2^s-1)}{2^{k_p(s+1)}(2^{s+1}-1)}+\frac{1}{2^{k_p(s+1)}(2^{s+1}-1)2^{m(s+1)}}\\
={} & \frac{2(2^s-1)}{2^{k_p(s+1)}(2^{s+1}-1)}+o(2^{-m}).
\end{align*}
If we now use that 
\begin{align*}
2(2^s-1)\sum_{i=1}^{m-1}\frac{1}{2^{(k_p+i)s}}\sum_{l=i+1}^{m}\frac{1}{2^{k_p+l}}
	={} & \frac{2(2^s-1)}{2^{(k_p+1)(s+1)}}\sum_{i=1}^{m-1}\frac{1}{2^{(i-1)s}}\left(\frac{1}{2^{i-1}}-\frac{1}{2^{m-1}}\right)\\
	={} &\frac{2(2^s-1)}{2^{k_p(s+1)}(2^{s+1}-1)}-\frac{2}{2^{k_p(s+1)+m}}+o(2^{-m}),
\end{align*}
we find (as $m\to\infty$)
\begin{align*}
 \frac{\mathcal{H}(x^+_m,s)-\mathcal{H}(x,s)}{x^+_m-x}	={} &\frac{\left(\frac{x^+_m}{x}\right)^{s+1}-1}{x^+_m-x}	\mathcal{H}(x,s)-\frac{(x^+_m)^{s}}{x}\left(2(2^{s}-1)\sum_{j=1}^{p-1}\frac{1}{2^{k_js}}+\frac{2}{2^{k_ps}}+o(1)\right),
\end{align*}
which after taking the limit as $m\to\infty$ gives us
\begin{align}
(s+1)\mathcal{H}(x,s)-x^s2(2^{s}-1)\sum_{j=1}^{p-1}\frac{1}{2^{k_js}}-\frac{2x^s}{2^{k_ps}}\geq {} &0,\qquad 0<s<1,\label{ecua21}\\
(s+1)\mathcal{H}(x,s)-x^s2(2^{s}-1)\sum_{j=1}^{p-1}\frac{1}{2^{k_js}}-\frac{2x^s}{2^{k_ps}}	\leq {} &0,  \qquad s>1.\label{ecua22}
\end{align}
Comparing \eqref{ecua19} with \eqref{ecua21} and \eqref{ecua20} with \eqref{ecua22} we deduce that
\begin{align*}
-\frac{2(2^{s}-1)}{2^{k_ps}}\leq {} & -\frac{2}{2^{k_ps}},\qquad 0<s<1,\\
-\frac{2(2^{s}-1)}{2^{k_ps}}\geq {} &-\frac{2}{2^{k_ps}}	,  \qquad s>1,
\end{align*}
both of which are impossible.

\section{Proof of Theorem \ref{approxthm}}

Recall that the numbers $x_{M,n}$ were introduced in \eqref{def:xMn}. For fixed values of $M$ and $n$, we have 
\begin{align*}
	\frac{1}{	x_{M,n}}=\frac{2^M+2n+1}{2^M}={} &\frac{\sum_{j=1}^{p}2^{n_j}}{2^M}=\sum_{k=1}^{p}\frac{1}{2^{M-n_j}},
\end{align*}	
where $M= n_1>n_2>\cdots>n_{p}=0$ and $p\geq 2$.  Writing $k_j=M-n_j$, we have  
\begin{align*}
	\frac{1}{	x_{M,n}}={} &\sum_{j=1}^{p}\frac{1}{2^{k_j}}
\end{align*}	
with 
$0= k_1<k_2<\cdots<k_{p}= M$, so that 
\begin{align*}
	\frac{1}{x_{M+1,2n+1}}={}&	\frac{2^{M+1}+4n+3}{2^{M+1}}=\sum_{j=1}^{p}\frac{1}{2^{k_j}}+\frac{1}{2^{M+1}}\\
	\frac{1}{x_{M+1,2n}}={} &	\frac{2^{M+1}+4n+1}{2^{M+1}}=\sum_{j=1}^{p-1}\frac{1}{2^{k_j}}+\frac{1}{2^{M+1}}.
\end{align*}
In particular, we see that  for all $s\geq -1$,
\begin{align}\label{ecua10}\begin{split}
		1-\left(\frac{x_{M+1,2n+1}}{x_{M,n}}\right)^{s+1}<{}
		&1-\left(1-\frac{	1}{2^{M+1}}\right)^{s+1}\leq \frac{s+1}{2^{M+1}},\\
		1-\left(\frac{x_{M,n}}{x_{M+1,2n}}\right)^{s+1}<{}
		&1-\left(1-\frac{	1}{2^{M+1}}\right)^{s+1}\leq \frac{s+1}{2^{M+1}}.
	\end{split}
\end{align}
Associated to the above binary representations we have the vectors 
\begin{align*}
	\vec{\vartheta}(	x_{M,n})={} &\left(	\frac{x_{M,n}}{2^{k_1}},\frac{x_{M,n}}{2^{k_2}},\ldots,\frac{	x_{M,n}}{2^{k_p}}\right),\\
	\vec{\vartheta}(	x_{M+1,2n+1})={} &\left(	\frac{x_{M+1,2n+1}}{2^{k_1}},\frac{x_{M+1,2n+1}}{2^{k_2}},\ldots,\frac{	x_{M+1,2n+1}}{2^{k_p}},\frac{	x_{M+1,2n+1}}{2^{M+1}}\right),\\
	\vec{\vartheta}(	x_{M+1,2n})={} &\left(	\frac{x_{M+1,2n}}{2^{k_1}},\frac{x_{M+1,2n}}{2^{k_2}},\ldots,\frac{	x_{M+1,2n}}{2^{k_{p-1}}},\frac{	x_{M+1,2n}}{2^{M+1}}\right).
\end{align*}
Since by definition, $\mathcal{H}(x,s)=	H(\vec{\vartheta}(x);s)$, plugging these vectors into \eqref{def:H} leads to the relations  
\begin{align}\label{ecua8}\begin{split}
		\mathcal{H}(x_{M,n},s)-	\mathcal{H}(x_{M+1,2n+1},s)	={} &\left(1-\left(\frac{x_{M+1,2n+1}}{x_{M,n}}\right)^{s+1}\right)\mathcal{H}(x_{M,n},s)\\
		&-(x_{M+1,2n+1})^{s+1}\left(\frac{1}{2^{(M+1)(s+1)}}	+\frac{2^{s}-1}{2^{M}}\sum_{j=1}^{p}\frac{1}{2^{k_js}}\right),
	\end{split}
\end{align}
\begin{align}\label{ecua9}\begin{split}
		\mathcal{H}(x_{M,n},s)	-\mathcal{H}(x_{M+1,2n},s)	={}&\left(\left(\frac{x_{M,n}}{x_{M+1,2n}}\right)^{s+1}-1\right)\mathcal{H}(x_{M+1,2n},s)\\
		&+(x_{M,n})^{s+1}\left(\frac{2^{s+1}-1}{2^{(M+1)(s+1)}}+\frac{2^{s}-1}{2^M}\sum_{j=1}^{p-1}\frac{1}{2^{k_js}}\right).
	\end{split}
\end{align}

Assume now that $0<s<1$, when we know that $0<d_s\leq \mathcal{H}(x_{M,n},s)	\leq 1$. From \eqref{ecua8} and \eqref{ecua10} we obtain	\begin{align}\label{ecua11}\begin{split}
		\mathcal{H}(x_{M,n},s)-	\mathcal{H}(x_{M+1,2n+1},s)	
		\leq {} & \frac{s+1}{2^{M+1}}.
	\end{split}
\end{align}
On the other hand, by the fact that $p\leq M+1$ we have the bound  
\[
\sum_{j=1}^{p-1}\frac{1}{2^{k_js}}\leq \sum_{\ell=0}^{M-1}\frac{1}{2^{\ell s}}=\frac{2^s}{2^s-1}\left(1-\frac{1}{2^{Ms}}\right), \qquad s\not=0,
\]
which together with \eqref{ecua9} yields  
\begin{align}\label{ecua15}\begin{split}
		\mathcal{H}(x_{M,n},s)	-\mathcal{H}(x_{M+1,2n},s)	\leq {}&\frac{2^{s+1}-1}{2^{(M+1)(s+1)}}+\frac{2^{s}-1}{2^M}\sum_{j=1}^{p-1}\frac{1}{2^{k_js}}\\
		\leq {}& \frac{1}{2^{M(s+1)}}- \frac{1}{2^{(M+1)(s+1)}}+\frac{2^{s}}{2^M}\left(1-\frac{1}{2^{sM}}\right)<\frac{2^s}{2^M}.
	\end{split}
\end{align}
Since $s+1<2^{s+1}$ for all $s>0$, we obtain from \eqref{ecua11} and \eqref{ecua15}  the inequality 
\begin{align*}
	\mathcal{H}(x_{M,n},s)\leq \min\{\mathcal{H}(x_{M+1,2n+1},s),\mathcal{H}(x_{M+1,2n},s)\}	+ \frac{2^s}{2^{M}}, \qquad n=0, 1,\ldots,2^{M-1}-1,
\end{align*}
which readily yields that $d_{s,M}\leq d_{s,M+1}	+2^s/2^{M}$. It follows that 
\begin{align*}
	0\leq d_{s,M}-	d_s={} & \lim_{L\to\infty}\sum_{\ell=0}^{L-1}(d_{s,M+\ell}-d_{s,M+\ell+1})\leq 2^s\sum_{\ell=0}^{\infty}\frac{1}{2^{M+\ell}}=\frac{2^s}{2^{M-1}}, 
\end{align*}	
proving \eqref{ecua12}.

Let us now assume that $s>1$ or that $-1<s<0$, in which case $d_s$ and $d_{s,M}$ are maximum values, and so  we have 
\begin{align}\label{ecua16}
	0\leq d_s- d_{s,M}={} &\lim_{L\to\infty}\sum_{\ell=0}^{L-1}(d_{s,M+\ell+1}-d_{s,M+\ell}).
\end{align}
If $s>1$  we deduce from \eqref{ecua8} that   	
\begin{align*}
	\mathcal{H}(x_{M+1,2n+1},s)-\mathcal{H}(x_{M,n},s)\leq {}	&	(x_{M+1,2n+1})^{s+1}\left(\frac{1}{2^{(M+1)(s+1)}}	+\frac{2^{s}-1}{2^{M}}\sum_{j=1}^{p}\frac{1}{2^{k_js}}\right)\\
	\leq {} &\frac{1}{2^{(M+1)(s+1)}}+\frac{2^s}{2^M}\left(1-\frac{1}{2^{s(M+1)}}\right)<\frac{2^s}{2^M}
\end{align*}
while from \eqref{ecua9}, \eqref{second-bound-H}, and \eqref{ecua10}  we deduce that
\begin{align*}
	\mathcal{H}(x_{M+1,2n},s)-\mathcal{H}(x_{M,n},s)		\leq {}&\left(1-\left(\frac{x_{M,n}}{x_{M+1,2n}}\right)^{s+1}\right)\mathcal{H}(x_{M+1,2n},s)\\
	\leq {}&\frac{1}{2^{M+1}}(2^{s+1}-2 )<\frac{2^s}{2^M}.
\end{align*}
Thus, it follows that for $s>1$.
\begin{align*}
	\max\{\mathcal{H}(x_{M+1,2n+1},s),\mathcal{H}(x_{M+1,2n},s)\}	\leq {} &\mathcal{H}(x_{M,n},s) 	+ \frac{2^s}{2^{M}}, \qquad n=0, 1,\ldots,2^{M-1}-1,
\end{align*}
which yields the inequality $d_{s,M+1}\leq d_{s,M}	+2^s/2^M$, and by means of \eqref{ecua16} we arrive at \eqref{ecua13}.

If $-1<s<0$, then $2^s-1<0$ and we get from \eqref{ecua8} that 
\begin{align}\label{ecua17}
	\mathcal{H}(x_{M+1,2n+1},s)	-	\mathcal{H}(x_{M,n},s)\leq \frac{1}{2^{(M+1)(s+1)}},
\end{align}
whereas \eqref{ecua9}, \eqref{ecua10} and \eqref{ecua6} yield
\begin{align*}
	\mathcal{H}(x_{M+1,2n},s)-	\mathcal{H}(x_{M,n},s)		\leq {}&\left(1-\left(\frac{x_{M,n}}{x_{M+1,2n}}\right)^{s+1}\right)\mathcal{H}(x_{M+1,2n},s)\\
	&-(x_{M,n})^{s+1}\frac{2^{s}-1}{2^M}\sum_{j=1}^{p-1}\frac{1}{2^{k_js}}\\
	\leq {} & \frac{(s+1)}{2^{M}}\frac{2^{s}}{2^{s+1}-1}+ \frac{2^s}{2^{M(s+1)}}\left(1-2^{sM}\right).
\end{align*}
Since 
\[
\frac{s+1}{2^{s+1}-1}< \frac{1}{\log 2},\qquad -1<s<0,
\]
we obtain 
\begin{align}\label{ecua18}
	\mathcal{H}(x_{M+1,2n},s)-	\mathcal{H}(x_{M,n},s)		\leq{}&  \frac{1}{2^{M(s+1)}}\left(\left(\frac{1}{\log 2}-1\right)2^{Ms}+1\right)\leq  \frac{1}{(\log 2)2^{M(s+1)}}.
\end{align}
The relations \eqref{ecua17}, \eqref{ecua18} and the fact that  
\[
\frac{1}{2^{(M+1)(s+1)}}<\frac{1}{(\log 2)2^{M(s+1)}}
\]
we have obtained that fo $-1<s<0$, allow us to conclude that for $-1<s<0$, we have 
\begin{align*}
	\max\{\mathcal{H}(x_{M+1,2n+1},s),\mathcal{H}(x_{M+1,2n},s)\}	\leq {} &\mathcal{H}(x_{M,n},s) 	+ \frac{1}{(\log 2)2^{M(s+1)}}  , \quad n=0, 1,\ldots,2^{M-1}-1.
\end{align*}
This inequality and \eqref{ecua16} quickly lead to \eqref{ecua14}.

\section{On the limit points of the sequences $G(\eta(N);s)$ and $\Lambda(\eta(N))$, and a Ces\`{a}ro summability property}\label{sect:density-optimal-potential}

Recall the definition of the functions $G(\vec{\theta};s)$ and $\Lambda(\vec{\theta})$ in \eqref{def:G-Lambda-functions}. In \cite{LopMin} it was proved that for any $s>0$, the set of limit points of the sequence $(G(\eta(N);s))_{N=1}^{\infty}$ is given by the closed interval with endpoints $1$ and $(2^{s}-1)^{-1}$, while the set of limit points of the sequence $(\Lambda(\eta(N)))_{N=1}^{\infty}$ coincides with the interval $[-2\log 2,0]$. From this result and \eqref{thirdcase} we deduced in \cite[Thm. 1.3]{LopMin} that the limit points of the sequence $(F_{N,s})_{N=1}^{\infty}$, $s>0$, fill a compact interval (the same claim is true in the case $s=0$). In this section we provide a new and simpler proof of the density results stated for the sequences $(G(\eta(N);s))_{N=1}^{\infty}$ and $(\Lambda(\eta(N)))_{N=1}^{\infty}$, based on the ideas developed in the present work. We also show that in the range $-2<s<0$, the sequence $(F_{N,s})_{N=1}^{\infty}$ is Ces\`{a}ro summable. 

First, we extend the definition \eqref{def:G-Lambda-functions}, and for $\vec{\vartheta}=(\vartheta_1,\vartheta_2,\ldots)\in \mathcal{S}$ we define 
\begin{align*}
G(\vec{\vartheta};s) & :=\sum_{k=1}^\infty\vartheta_k^s,\qquad s>0,\\
\Lambda(\vec{\vartheta}) & :=\sum_{k=1}^{\infty}\vartheta_{k}\log \vartheta_{k}.
\end{align*}
Using the Lebesgue dominated convergence theorem, it is easy to see that $G(\vec{\vartheta};s)$ and $\Lambda(\vec{\vartheta})$ are continuous over $\mathcal{S}$. Note also that $G(\vec{\vartheta};1)=1$ for all $\vec{\vartheta}\in\mathcal{S}$.

For $x\in [1/2, 1]$, we define $\mathcal{G}(x,s)$, $s>0$, and $\widetilde{\Lambda}(x)$ as follows. If $x$ is such that $1/x$ can be written in the form \eqref{15-8-1}, we write 
\begin{align*}
\mathcal{G}(x,s) & :=G(\vec{\vartheta}(x);s),\\
\widetilde{\Lambda}(x) & :=\Lambda(\vec{\vartheta}(x)),
\end{align*}
otherwise we set 
\begin{align*}
\mathcal{G}(x,s) & :=	G(\vec{\vartheta}_\infty(x);s),\\
\widetilde{\Lambda}(x) & :=\Lambda(\vec{\vartheta}_{\infty}(x)),
\end{align*}
see \eqref{firstform} and \eqref{secondform}. 

\begin{proposition}\label{prop_rangeGcal}
The range of the functions $\mathcal{G}(x,s)$ and $\widetilde{\Lambda}$ are the closed intervals
\begin{align}
\mathcal{G}([1/2,1],s) & =[1,(2^{s}-1)^{-1}],\qquad 0<s<1,\label{rangeGsl1}\\
\mathcal{G}([1/2,1],s) & =[(2^{s}-1)^{-1},1],\qquad s>1,\label{rangeGsb1}\\
\widetilde{\Lambda}([1/2,1]) & =[-2\log 2 ,0].\label{rangeLambdatilde}
\end{align}
\end{proposition}
\begin{proof}
It was (essentially) proven in \cite{LopMin} that  
\begin{align}
1 & \leq G(\vec{\vartheta};s) \leq \frac{1}{2^s-1},\qquad 0<s<1,\label{boundG1}\\
\frac{1}{2^{s}-1} & \leq G(\vec{\vartheta};s)\leq 1,\qquad s>1,\label{boundG2}\\
-2\log 2 & \leq \Lambda(\vec{\vartheta})\leq 0,\label{boundLambda}
\end{align}
where the upper and lower bounds are attained (formulas (1.24)--(1.26) in \cite{LopMin} show that the above inequalities are valid for vectors $\vec{\vartheta}=\eta(N)$, $N\in\mathbb{N}$, and therefore by continuity they are valid for $\vec{\vartheta}\in\mathcal{S}$). The bounds $1$ and $0$ are attained at $(1,0,0,\ldots)$, and the bounds $(2^{s}-1)^{-1}$ and $-2\log 2$ are attained at $(2^{-1},2^{-2},\ldots, 2^{-n},\ldots)$.

Assume first that $0<s<1$. For all $x\in[1/2,1]$ we have
\[
1=\mathcal{G}(1,s) \leq \mathcal{G}(x,s)\leq \mathcal{G}(1/2,s)=\frac{1}{2^s-1}.\]
By the continuity of $G$ over $\mathcal{S}$, using Lemma \ref{lem_elem1} we see that if $1/2<x<1$ has two associated vectors $\vec{\vartheta}(x)$ and $\vec{\vartheta}_\infty(x)$, then 
\[
\lim_{y\to x-}\mathcal{G}(y,s)=G(\vec{\vartheta}(x);s),\qquad \lim_{y\to x+}\mathcal{G}(y,s)=G(\vec{\vartheta}_\infty(x);s).
\]
Because $0<s<1$, we get from the relation \eqref{relation-two-vectors} between the vectors $\vec{\vartheta}(x)$ and $\vec{\vartheta}_\infty(x)$ that
\[
G(\vec{\vartheta}_\infty(x);s)-G(\vec{\vartheta}(x);s)=\frac{x^s}{2^{k_ms}}\left(\sum_{j=1}^\infty\frac{1}{2^{sj}}-1\right)=\frac{x^s}{2^{k_ms}}\left(\frac{2-2^s}{2^s-1}\right)>0.
\]
Thus for an $x$ with two associated vectors we have 
\[
\lim_{y\to x-}\mathcal{G}(y,s)=\mathcal{G}(x,s),\qquad \lim_{y\to x+}\mathcal{G}(y,s)>\mathcal{G}(x,s).
\]
From Lemma \ref{lem_elem1} we also obtain that $\lim_{y\rightarrow 1-}\mathcal{G}(y,s)=\mathcal{G}(1,s)$. Using Lemma \ref{lem_elem2} we see that if $x$ has only one associated vector  $\vec{\vartheta}_\infty(x)$, then $\lim_{y\to x}\mathcal{G}(y,s)=\mathcal{G}(x,s)$. 

Summarizing, we have found that for all $x\in[1/2,1]$,
\[
\lim_{y\to x-}\mathcal{G}(y,s)=\mathcal{G}(x,s),\qquad \lim_{y\to x+}\mathcal{G}(y,s)\geq \mathcal{G}(x,s).
\]
Using these relations we can now prove \eqref{rangeGsl1}. Suppose $v\in (1,(2^s-1)^{-1})$, so that 
\[
\mathcal{G}(1,s)<v<\mathcal{G}(1/2,s),
\]
and let $x_v$ be the supremum of the set $
I_v=\{y\in [1/2,1]: \mathcal{G}(y,s)>v\}$. Since $1/2\in I_v$, $x_v$ exists and we have 
\[
\mathcal{G}(x_v,s)=\lim_{y\to x_v-}\mathcal{G}(y,s)\geq v.
\]
This implies that $x_v<1$. If, however, $\mathcal{G}(x_v,s)>v$, then because  $\lim_{y\to x_v+}\mathcal{G}(y,s)\geq \mathcal{G}(x_v,s)$, there must exist some $y_0\in (x_v,1)$ such that $\mathcal{G}(y_0,s)>v$, contradicting the supremum definition of $x_v$. The conclusion is therefore that $\mathcal{G}(x_v,s)=v$.

Reversing the inequalities in the above argument, one can prove \eqref{rangeGsb1}. Likewise, using the continuity of $\Lambda$ over $\mathcal{S}$, reversing the inequalities in the above argument one obtains \eqref{rangeLambdatilde} (note that in this case we have $\Lambda(\vec{\vartheta}_{\infty}(x))-\Lambda(\vec{\vartheta}(x))=-\log (4) x 2^{-k_{m}}<0$). 
\end{proof}

Analogously to Theorem \ref{theo:limpointHKR}, by the continuity of the functions $G(\cdot;s)$ and $\Lambda(\cdot)$ over $\mathcal{S}$, we obtain that the sets of all limit points of the sequences $(G(\eta(N);s))_{N=1}^{\infty}$ and $(\Lambda(\eta(N)))_{N=1}^{\infty}$ are identical to the sets $\{G(\vec{\vartheta};s): \vec{\vartheta}\in\mathcal{S}\}$ and $\{\Lambda(\vec{\vartheta}): \vec{\vartheta}\in\mathcal{S}\}$, respectively. We also have the identities
\begin{align*}
\{G(\vec{\vartheta};s): \vec{\vartheta}\in\mathcal{S}\} & =\mathcal{G}([1/2,1],s),\qquad s>0,\\
\{\Lambda(\vec{\vartheta}): \vec{\vartheta}\in\mathcal{S}\} & =\widetilde{\Lambda}([1/2,1]).
\end{align*}
Indeed, it is obvious from the definition of the functions $\mathcal{G}$ and $\widetilde{\Lambda}$ that $\mathcal{G}([1/2,1],s)\subseteq \{G(\vec{\vartheta};s): \vec{\vartheta}\in\mathcal{S}\}$ and $\widetilde{\Lambda}([1/2,1])\subseteq \{\Lambda(\vec{\vartheta}): \vec{\vartheta}\in\mathcal{S}\}$. The other inclusion follows immediately from \eqref{boundG1}--\eqref{boundLambda} and Proposition \ref{prop_rangeGcal}. In conclusion, the interval with endpoints $1$ and $(2^{s}-1)^{-1}$ coincides with the set of limit points of the sequence $(G(\eta(N);s))_{N=1}^{\infty}$, and $[-2\log 2,0]$ is the set of limit points of the sequence $(\Lambda(\eta(N)))_{N=1}^{\infty}$.

In \cite{LopMc1}, Theorem 3.11, it was proved that in the range $-2<s<0$, the sequence 
\begin{equation}\label{6-8-1}
(U_{N,s}(a_{N})-N I_{s}(\sigma))_{N=1}^{\infty}
\end{equation}
is bounded and divergent. In fact, for each $N\geq 1$, the inequality
\[
0<U_{N,s}(a_{N})-N I_{s}(\sigma)<I_{s}(\sigma)
\]
holds, and the bounds are sharp. In our next result we show that the sequence \eqref{6-8-1} is, however, Ces\`{a}ro summable.

\begin{theorem}
Let $-2<s<0$, and let $(a_{n})_{n=0}^{\infty}\subset S^{1}$ be a greedy $s$-energy sequence. Then
\[
\lim_{N\rightarrow\infty}\frac{1}{N}\sum_{k=1}^{N}(U_{k,s}(a_{k})-k I_{s}(\sigma))=\frac{I_{s}(\sigma)}{2}.
   \]
In fact, we have
\[
\frac{1}{N}\sum_{k=1}^{N}(U_{k,s}(a_{k})-k I_{s}(\sigma))-\frac{I_{s}(\sigma)}{2}=\begin{cases}
O(N^{s}), & -1<s<0,\\
O(N^{-1}\log N), &\,\,\,\,\,\,\,\, s=-1,\\
O(N^{-1}), & -2<s<-1.
\end{cases}
\]
\end{theorem}
\begin{proof}
First observe that
\begin{align*}
\sum_{k=1}^{N}(U_{k,s}(a_{k})-k I_{s}(\sigma)) & =\sum_{k=1}^{N}U_{k,s}(a_{k})-\frac{N(N+1)}{2} I_{s}(\sigma)\\
& =\frac{1}{2}E_{s}(\alpha_{N+1,s})-\frac{N(N+1)}{2} I_{s}(\sigma).
  \end{align*}
It follows that
\[
\frac{1}{N}\sum_{k=1}^{N}(U_{k,s}(a_{k})-k I_{s}(\sigma))=\frac{1}{2N}(E_{s}(\alpha_{N+1,s})-(N+1)^{2}I_{s}(\sigma))+\frac{1}{2}\frac{N+1}{N} I_{s}(\sigma).
\]
We know by Theorems 3.16, 3.17, and 3.18 in \cite{LopMc1} that the following sequences are bounded:
\begin{gather*}
\frac{E_{s}(\alpha_{N,s})-N^{2}I_{s}(\sigma)}{N^{1+s}},\qquad -1<s<0,\\
\frac{E_{-1}(\alpha_{N,-1})-N^{2}I_{-1}(\sigma)}{\log N},\qquad s=-1,\\
E_{s}(\alpha_{N,s})-N^{2}I_{s}(\sigma),\qquad -2<s<-1.
\end{gather*}
Therefore
\[
\lim_{N\rightarrow\infty}\frac{1}{N}(E_{s}(\alpha_{N+1,s})-(N+1)^{2}I_{s}(\sigma))=0
\]
and the claims follow.
\end{proof}

\bigskip

\noindent \textsc{School of Data, Mathematical, and Statistical Sciences, University of Central Florida, 4393 Andromeda Loop North, Orlando, FL 32816, USA} \\
\textit{Email address}: \texttt{abey.lopez-garcia\symbol{'100}ucf.edu}

\bigskip

\noindent \textsc{Department of Mathematics, The University of Mississippi, Hume Hall 305, P.O. Box 1848, University, MS 38677, USA}\\
\textit{Email address}: \texttt{minadiaz\symbol{'100}olemiss.edu}

\end{document}